\documentclass[12pt, a4paper]{amsart}
\usepackage[hmargin=32mm, vmargin=27mm, includefoot, twoside]{geometry}
\usepackage[bookmarksopen=true]{hyperref}
\usepackage[latin1]{inputenc} 
\usepackage[english]{babel}

\usepackage{amssymb}
\usepackage{latexsym}
\usepackage{mathrsfs}
\usepackage{xspace}
\usepackage{latexsym,mathrsfs,xspace, graphicx}
\usepackage[usenames,dvipsnames]{color}
\usepackage{enumerate}
\usepackage{lmodern}
\usepackage[all]{xy}
\usepackage{xr}

\newtheorem{thm}{Theorem}[section]
\newtheorem{cor}[thm]{Corollary}
\newtheorem{prop}[thm]{Proposition}
\newtheorem{lemma}[thm]{Lemma}

\newtheorem{claim}{Claim}
\newtheorem{thmintro}{Theorem}

\newtheorem{corintro}[thmintro]{Corollary}

\theoremstyle{definition}

\newtheorem{rem}[thm]{Remark}

\theoremstyle{remark}

\newcommand{\FF}{\mathbf{F}}

\newcommand{\NN}{\mathbf{N}}
\newcommand{\QQ}{\mathbf{Q}}

\newcommand{\ZZ}{\mathbf{Z}}
\newcommand{\GGG}{\mathcal{G}}
\newcommand{\AAA}{\mathcal{A}}

\newcommand{\inv}{^{-1}}
\newcommand{\la}{\langle}
\newcommand{\ra}{\rangle}

\def\co{\colon\thinspace}

\DeclareMathOperator{\E}{E}
\DeclareMathOperator{\Fix}{Fix}
\DeclareMathOperator{\Stab}{Stab}
\DeclareMathOperator{\dist}{d}
\DeclareMathOperator{\Pc}{Pc}

\DeclareMathOperator{\Aut}{Aut}

\DeclareMathOperator{\Ess}{Ess}
\DeclareMathOperator{\CAT}{CAT(0)}

\DeclareMathOperator{\Min}{Min}
\usepackage{xr}
\numberwithin{equation}{section}

\begin{document}

\renewcommand{\proofname}{{\bf Proof}}

\title[Open subgroups of Kac--Moody groups]{Open subgroups of locally compact Kac--Moody groups}
\author[P-E.~Caprace]{Pierre-Emmanuel \textsc{Caprace}$^*$}
\address{UCL, 1348 Louvain-la-Neuve, Belgium}
\email{pe.caprace@uclouvain.be}
\thanks{$^*$F.R.S.-FNRS Research Associate, supported in part by FNRS grant F.4520.11}

\author[T.~Marquis]{Timoth\'ee \textsc{Marquis}$^{\ddagger}$}
\address{UCL, 1348 Louvain-la-Neuve, Belgium}
\email{timothee.marquis@uclouvain.be}
\thanks{$^\ddagger$F.R.S.-FNRS Research Fellow}

\date{August 2011}
%\subjclass[2010]{20E42, 20F55, 20G44}
%\keywords{Kac--Moody group, locally compact group, open subgroup, Noetherian group, BN-pair, Coxeter group, parabolic closure}
%
\begin{abstract}
Let $G$ be a complete Kac-Moody group over a finite field. It is known that $G$ possesses a BN-pair structure, all of whose parabolic subgroups are open in $G$. We show that, conversely, every open subgroup of $G$ is contained with finite index in some parabolic subgroup; moreover there are only finitely many such parabolic subgroups. The proof uses some new results on parabolic closures in Coxeter groups. In particular, we give conditions ensuring that  the parabolic closure of the product of two elements in a Coxeter group contains the respective  parabolic closures of those elements.
\end{abstract}

\maketitle

%%%%%%%%%%%%%%%%%%%%%%%%%%%%%%%%%%%%%%%%%%%%%%%%%
%%%%%%%%%%%%%%%%%%%%%%%%%%%%%%%%%%%%%%%%%%%%%%%%%
\section{Introduction}
%%%%%%%%%%%%%%%%%%%%%%%%%%%%%%%%%%%%%%%%%%%%%%%%%
%%%%%%%%%%%%%%%%%%%%%%%%%%%%%%%%%%%%%%%%%%%%%%%%%

This paper is devoted to the study of open subgroups of complete Kac--Moody groups over finite fields. The interest in the structure of those groups is motivated by the fact that they constitute a prominent family of locally compact groups which are simultaneously \emph{topologically simple} and \emph{non-linear over any field} (see \cite{RemyGAFA})  and \cite{CaRe}). They show some resemblance with the simple linear locally compact groups arising from semi-simple algebraic groups over local fields of positive characteristic.

The first question on open subgroups of a given locally compact group $G$ one might ask is: How many such subgroups are there? Let us introduce some terminology providing possible answers to this question. We say that $G$  \textbf{has few open subgroups}  if every proper open subgroup of $G$ is compact. We say that $G$ is \textbf{Noetherian} if $G$ satisfies an ascending chain condition on open subgroups. Equivalently $G$ is Noetherian if and only if every open subgroup of $G$ is compactly generated (see Lemma~\ref{lem:Noeth} below). Clearly, if $G$ has few open subgroups, then it is Noetherian. Basic examples of locally compact groups that are Noetherian --- and in fact, even have few open subgroups --- are connected groups and compact groups. Noetherianity can thus be viewed as a finiteness condition which generalizes simultaneously the notion of connectedness and of compactness. It is highlighted in \cite{CaMo}, where it is notably shown that a Noetherian group admits a subnormal series with every subquotient  compact, or abelian, or simple. An example of a non-Noetherian group is given by the additive group $\QQ_p$ of the $p$-adics. Other examples, including   simple ones, can be constructed as groups acting on trees. 

According to a theorem of G.~Prasad~\cite{Prasad} (which he attributes to Tits), simple locally compact groups arising from algebraic groups over local fields have few open subgroups. Locally compact Kac--Moody groups are however known to have a broader variety of open subgroups in general. Indeed, Kac--Moody groups are  equipped with a $BN$-pair all of whose parabolic subgroups are open. In particular, if the Dynkin diagram of a Kac--Moody group admits proper subdiagrams that are not of spherical type, then the corresponding Kac--Moody groups have proper open subgroups that are not compact.

Our main result is   that parabolic subgroups  in Kac--Moody groups are  essentially the only source of open subgroups. 

\begin{thmintro}\label{main thm intro}
Every open subgroup of a complete Kac-Moody group $G$ over a finite field has finite index in some parabolic subgroup. 

Moreover, given an open subgroup $O$, there are only finitely many distinct parabolic subgroups of $G$ containing $O$ as a finite index subgroup.
\end{thmintro}

A more precise statement of this theorem will be given later, see Theorem~\ref{thm complet}.
As a consequence, we deduce the following.

\begin{corintro}\label{cor intro}
Complete Kac--Moody groups over finite fields are Noetherian.
\end{corintro}

In fact, Theorem~\ref{main thm intro} allows us to characterize those locally compact Kac--Moody groups having few open subgroups, as follows. 

\begin{corintro}\label{corintro2}
Let $G$ be a complete Kac--Moody group of irreducible type over a finite field. Then $G$ has few open subgroups if and only if the Weyl group of $G$ is of affine type, or of compact hyperbolic type.
\end{corintro}

Notice that the list of all compact hyperbolic types of Weyl groups is finite  and contains diagrams of rank at most~$5$ (see \emph{e.g.}  Exercise V.4.15 on p.~133 in \cite{Bourbaki}). The groups in Corollary~\ref{corintro2} include in particular all complete Kac--Moody groups of rank two.

Another application of Theorem~\ref{main thm intro} is that it shows how the $BN$-pair structure is encoded in the topological group structure of a Kac--Moody group. Here is a precise formulation of this.

\begin{corintro}
Let $G$ be a complete Kac--Moody group over a finite field and $P<G$ be an open subgroup. 
If $P$ is maximal in its commensurability class, then $P$ is a parabolic subgroup of $G$. 
\end{corintro}

Our proof of Theorem~\ref{main thm intro} relies on some new results on parabolic closures in Coxeter groups, which we now proceed to describe. Let thus $(W, S)$ be a Coxeter system with $S$ finite. Recall that any intersection of parabolic subgroups in $W$ is itself a parabolic subgroup. Following D.~Krammer~\cite{MR2466021}, it thus makes sense to define the \textbf{parabolic closure} of a subset of  $W$ as the intersection of all parabolic subgroups containing it. The parabolic closure of a set $E \subseteq W$ is denoted by $\Pc(E)$. 

\begin{thmintro}\label{thm:Cox1}
Let $w \in W$ be an element of infinite order and let $\lambda$ be a translation axis for $w$ in  the Davis complex. Assume that the parabolic closure $\Pc(w)$ is of irreducible type. 

Then there is a constant $C$ such that for any two parallel walls $m, m'$ transverse to $\lambda$, if $d(m, m') > C$, then $\Pc(w) =  \Pc(r_{m}, r_{m'}) $.

\end{thmintro}

In particular, we get the following.

\begin{corintro}
Any irreducible non-spherical parabolic subgroup of a Coxeter group is the parabolic closure of a pair of reflections. 
\end{corintro}

Our main result on Coxeter groups concerns the parabolic closure of the product of two elements.

\begin{thmintro}\label{thm:Cox2}
There is a finite index normal subgroup $W_0 < W$ enjoying the following property.

For all $g,h\in W_0$, there exists a constant $K=K(g,h)\in\NN$ such that for all $m,n\in\ZZ$ with $\min\{|m|,|n|,|m/n|+|n/m|\}\geq K$, we have $\Pc(g^m h^n) \supseteq \Pc(g) \cup \Pc(h)$.
\end{thmintro}

The following corollary is an essential ingredient in the proof of Theorem~\ref{main thm intro}.

\begin{corintro}\label{cor:fundamental}
Let $H$ be a subgroup of $W$. Then there exists $h\in H$ such that the parabolic closure of $h$ has finite index in the parabolic closure of $H$.
\end{corintro}

%%%%%%%%%%%%%%%%%%%%%%%%%%%%%%%%%%%%%%%%%%%%%%
%%%%%%%%%%%%%%%%%%%%%%%%%%%%%%%%%%%%%%%%%%%%%%
\section{Walls and parabolic closures in  Coxeter groups}
%%%%%%%%%%%%%%%%%%%%%%%%%%%%%%%%%%%%%%%%%%%%%%
%%%%%%%%%%%%%%%%%%%%%%%%%%%%%%%%%%%%%%%%%%%%%%

Throughout this section, we let $(W,S)$ be a Coxeter system with $W$ finitely generated (equivalently $S$ is
finite). Let $\Sigma$ be the associated Coxeter complex, and let $|\Sigma|$ denote its standard geometric realization. Also, let $X$ be the Davis realization of $\Sigma$. Thus $X$ is a $\CAT$ subcomplex of the barycentric subdivision of $|\Sigma|$.

Let $\Phi=\Phi(\Sigma)$ denote the set of half-spaces of $\Sigma$. A half-space $\alpha \in \Phi$ will also be called a \textbf{root}. Given a root $\alpha\in\Phi$, we write $r_{\alpha} = r_{\partial \alpha}$ for the unique reflection of $W$ fixing the wall $\partial\alpha$ of $\alpha$ pointwise. 

We say that two walls $m, m'$ of  $X$ are \textbf{parallel} if either they coincide or they are disjoint. We say that the walls $m, m'$ are \textbf{perpendicular} if they are distinct and if the reflections $r_{m}$ and $r_{m'}$ commute.

Finally, for a subset $J\subseteq S$, we set $J^{\perp}:=\{s\in S\setminus J \ | \ sj=js \ \forall j\in J\}$. 

In this paper, we call a subset $J\subseteq S$ \textbf{essential} if each irreducible component of $J$ is non-spherical.

%%%%%%%%%%%%%%%%%%%%%%%%%%%%%%%%%%%%%%%%%%%%%%%%%%%%%%
\subsection{The normalizer of a parabolic subgroup}
%%%%%%%%%%%%%%%%%%%%%%%%%%%%%%%%%%%%%%%%%%%%%%%%%%%%%%

\begin{lemma}\label{lemme Deodhar}
Let $L\subseteq S$ be essential. Then $\mathscr{N}_W(W_L)=W_L\times \mathscr{Z}_W(W_L)$ and is again parabolic. Moreover, $\mathscr{Z}_W(W_L)=W_{L^{\perp}}$.
\end{lemma}
\begin{proof}
See \cite[Proposition 5.5]{MR647210} and \cite[Chapter~3]{MR2466021}.
\end{proof}

%%%%%%%%%%%%%%%%%%%%%%%%%%%%%%%%%%%%%%%%%%%%%%%%
\subsection{Preliminaries on parabolic closures}
%%%%%%%%%%%%%%%%%%%%%%%%%%%%%%%%%%%%%%%%%%%%%%%%

A subgroup of $W$ of the form $W_J$ for some $J \subset S$ is called a \textbf{standard parabolic
subgroup}. Any of its conjugates is called a \textbf{parabolic subgroup} of $W$. Since  any intersection of parabolic subgroups is itself a parabolic subgroup (see \cite{MR0470099}), it makes sense to define the
\textbf{parabolic closure} $\Pc(E)$ of a subset $E \subset W$ as the smallest parabolic subgroup of $W$
containing $R$. For $w\in W$, we will also write $\Pc(w)$ instead of $\Pc(\{w\})$.

\begin{lemma}\label{lem:PC}
Let $G$ be a reflection subgroup of $W$, namely a subgroup of $W$ generated by a set $T$ of reflections. We have
the following:
\begin{itemize}
\item[(i)] There is a set of reflections $R \subset G$, each conjugate to some element of $T$, such that $(G, R)$ is a Coxeter system.

\item[(ii)] If $T$ has no nontrivial partition $T = T_1 \cup T_2$ such that $[T_1, T_2]= 1$, then $(G, R)$ is
irreducible.

\item[(iii)] If $(G, R)$ is irreducible (resp. spherical, affine of rank $\geq 3$), then so is $\Pc(G)$.

\item[(iv)] If $G'$ is a reflection subgroup of irreducible type which centralizes $G$ and if $G$ is of
irreducible non-spherical type, then either $\Pc(G \cup G') \cong \Pc(G) \times \Pc(G')$ or $\Pc(G) = \Pc(G')$
is of irreducible affine type.
\end{itemize}
\end{lemma}
\begin{proof}
For (i) and (iii), see \cite[Lemma~2.1]{MR2665193}. Assertion (ii) is easy to verify. For (iv), see
\cite[Lemma~2.3]{MR2665193}.
\end{proof}

\begin{lemma}\label{lem:8walls}
Let $\alpha_0 \subsetneq \alpha_1 \subsetneq \dots \subsetneq \alpha_k$ be a nested sequence of half-spaces such
that $A = \la r_{\alpha_i} \; | \; i = 0, \dots, k\ra$ is infinite dihedral. If $k \geq 7$, then for any wall
$m$ which meets every $\partial \alpha_i$, either $r_m$ centralizes $\Pc(A)$, or  $\la A \cup \{r_m\} \ra$ is a
Euclidean triangle group.
\end{lemma}
\begin{proof}
This follows from \cite[Lemma~11]{MR2263057} together with Lemma~\ref{lem:PC}(iv).
\end{proof}

%%%%%%%%%%%%%%%%%%%%%%%%%%%%%%%%%%%%%%%%%%%
\subsection{Parabolic closures and finite index subgroups}
%%%%%%%%%%%%%%%%%%%%%%%%%%%%%%%%%%%%%%%%%%%

\begin{lemma}\label{lemme parabolique indice fini}
Let $H_1 < H_2$ be subgroups of $W$. If $H_1$ is of finite index in $H_2$, then $\Pc(H_1)$ is of finite index in $\Pc(H_2)$.
\end{lemma}
\begin{proof}
For $i=1,2$, set $P_i:=\Pc(H_i)$. Since the kernel $N$ of the action of $H_2$ on the coset space $H_2/H_1$ is a finite index normal subgroup of $H_2$ that is contained in $H_1$, so that in particular $\Pc(N)\subseteq\Pc(H_1)$, we may assume without loss of generality that $H_1$ is normal in $H_2$. But then $H_2$ normalizes $P_1$. Up to conjugating by an element of $W$, we may also assume that $P_1$ is standard, namely $P_1=W_I$ for some $I\subseteq S$. Finally, it is sufficient to prove the lemma when $I$ is essential, which we assume henceforth. Lemma~\ref{lemme Deodhar} then implies that $P_2<W_I\times W_{I^{\perp}}$. We thus have an action of $H_2$ on the residue $W_I\times W_{I^{\perp}}$, and since $H_1$ stabilizes $W_I$ and has finite index in $H_2$, the induced action of $H_2$ on $W_{I^{\perp}}$ possesses finite orbits. By the Bruhat--Tits fixed point theorem (see for example \cite[Th.11.23]{ABrown}), it follows that $H_2$ fixes a point in the Davis realization of $W_{I^{\perp}}$, that is, it stabilizes a spherical residue of $W_{I^{\perp}}$. This shows $[P_2:W_I]<\infty$.
\end{proof}

%%%%%%%%%%%%%%%%%%%%%%%%%%%%%%%%%%%%%%%%%%%%%%%%%%%%%%
\subsection{Parabolic closures and essential roots}
%%%%%%%%%%%%%%%%%%%%%%%%%%%%%%%%%%%%%%%%%%%%%%%%%%%%%%

 Our next goal is to present a description of the parabolic closure $\Pc(w)$ of an element $w \in W$, which is essentially due to D.~Krammer~\cite{MR2466021}.

\medskip
Let $w\in W$. A root $\alpha\in\Phi$ is called  \textbf{$w$-essential} if either $w^n\alpha\subsetneq \alpha$ or $w^{-n}\alpha\subsetneq \alpha$ for some $n>0$. A wall is called $w$-essential if it bounds a $w$-essential root. We denote by 
$$
\Ess(w)
$$ 
the set of $w$-essential walls. Clearly $\Ess(w)$ is empty if $w$ is of finite order. If $w$ is of infinite order, then it acts on $X$ as a hyperbolic isometry and thus possesses some translation axis. We say that a wall is \textbf{transverse} to such an axis if it intersects this axis in a single point. We recall that the intersection of a wall and any geodesic segment which is not completely contained in that wall is either empty or consists of a single point (see \cite[Lemma~3.4]{MR1885045}). Given $x, y \in X$, we say
that a wall $m$ \textbf{separates}  $x$ from $y$ if the intersection $[x, y] \cap m$ consists of a single point.

\begin{lemma}\label{lem:essential=transverse}
Let $w \in W$ be of infinite order and let $\lambda $ be a translation axis for $w$ in $X$. 
Then $\Ess(w)$ coincides with those walls which are transverse to $\lambda$. 
\end{lemma}

The proof requires a subsidiary fact. Recall that Selberg's lemma ensures that any finitely generated linear group over $\mathbf C$ admits a finite index torsion-free subgroup. This is thus the case for Coxeter groups. The following lemma provides important combinatorial properties of those torsion-free subgroups of Coxeter groups.  Throughout the rest of this section, we let $W_0 < W$ be a torsion-free finite index normal subgroup.

\begin{lemma}\label{lemme DJ}
For all $w \in W_0$ and  $\alpha \in \Phi$,   either $w\alpha= \alpha$ or $w.\partial \alpha \cap \partial \alpha = \varnothing$. 
\end{lemma}

\begin{proof}
See Lemma~1 in   \cite{DraJa}. 
\end{proof}

\begin{proof}[Proof of Lemma~\ref{lem:essential=transverse}]
It is clear that if $\alpha \in \Phi$ is $w$-essential, then $\partial \alpha$ is tranverse to any $w$-axis. To see the converse, let $n >0$  be such that $w^n \in W_0$. Since $\lambda$ is also a $w^n$-axis, we deduce from Lemma~\ref{lemme DJ} that for all roots $\alpha$ such that $\partial \alpha$ is transverse to $\lambda$, we have either $w^n \alpha \subsetneq \alpha$ or $\alpha \subsetneq w^n\alpha$. The result follows.
\end{proof}

 We also set 
$$
\Pc^\infty(w)=\langle r_{\alpha} \ | \ \textrm{$\alpha$ is a $w$-essential root}\rangle.
$$

\medskip

Notice that every nontrivial element of $W_0$ is hyperbolic. Moreover, in view of Lemma~\ref{lemme DJ}, we deduce that if $w \in W_0$, then a root $\alpha$ is $w$-essential if and only if $w\alpha\subsetneq \alpha$ or $w^{-1}\alpha\subsetneq \alpha$.

\begin{lemma}\label{lemme Krammer}
Let $w \in W$ be of infinite order, let $\lambda $ be a translation axis for $w$ in $X$ and let $x\in \lambda$.

Then we have the following. 
\begin{enumerate}[(i)]
\item 
$\begin{array}[t]{lcl}
\Pc^\infty(w) 
&= & \langle r_{\alpha} \ | \ \textrm{$\partial\alpha$ is a wall transverse to $\lambda$}\rangle\\
&= & \langle r_{\alpha} \ | \ \textrm{$\partial\alpha$ is a wall transverse to $\lambda$ and separates $x$ from $wx$}\rangle.
\end{array}$

\item  $\Pc^\infty(w)$ coincides with the essential component of $\Pc(w)$, \emph{i.e.} the product of its non-spherical components. In particular $\Pc(w) = \Pc^\infty(w)$ if and only if $\Pc(w)$ is of essential type. 

\item If $w \in W_0$, then $\Pc(w) = \Pc^\infty(w)$. 
\end{enumerate}
\end{lemma}
\begin{proof}
The first equality in Assertion~(i) follows from Lemma~\ref{lem:essential=transverse}. To check the second, it suffices to remark that if $\partial\alpha$ is any wall transverse to $\lambda$, then there exists a power $w^k$ of $w$ such that $w^k \partial\alpha$ separates $x$ from $wx$.

Assertion~(ii)  follows from  Corollary~5.8.7 in \cite{MR2466021} (notice that what we call \emph{essential} roots here are called \emph{odd} roots in \emph{loc.~cit.}). 
Assertion~(iii) follows from Lemma~\ref{lemme DJ} and {Theorem~5.8.3} from \cite{MR2466021}. 
\end{proof}

%%%%%%%%%%%%%%%%%%%%%%%%%%%%%%%%%%%%%%%%%%%
\subsection{The Grid Lemma}
%%%%%%%%%%%%%%%%%%%%%%%%%%%%%%%%%%%%%%%%%%%

The following lemma is an unpublished observation due to the first author and Piotr Przytycki.

\begin{lemma}[Caprace--Przytycki]\label{lemme grid lemma}
There exists a constant $N$, depending only on $(W,S)$, such that the following property holds. Let $\alpha_0
\subsetneq \alpha_1 \subsetneq \dots \alpha_k$ and $\beta_0 \subsetneq \beta_1 \subsetneq \dots \subsetneq
\beta_l$ be two nested families of half-spaces of $X$ such that $\min\{k, l\} > 2 N$. Set $A = \la r_{\alpha_i}
\; | \; i=0, \dots, k\ra $, $A' = \la r_{\alpha_i} \; | \; i=N, N+1, \dots, k-N\ra $, $B =  \la r_{\beta_j} \; |
\; j=0, \dots, l\ra $ and $B' =  \la r_{\beta_j} \; | \; j= N, N+1, \dots, k-N\ra $. If $\partial \alpha_i$
meets $\partial \beta_j$ for all $i, j$, then either of the following assertions holds:
\begin{itemize}
\item[(i)] The groups $A$ and $B$ are both infinite dihedral, their union generates a Euclidean triangle group
and the parabolic closure $\Pc(A \cup B)$ coincides with $\Pc(A)$ and $\Pc(B)$ and is of irreducible affine
type.

\item[(ii)] The parabolic closures $\Pc(A)$, $\Pc(A')$, $\Pc(B)$ and $\Pc(B')$ are all of irreducible type;
furthermore we have
$$\Pc(A' \cup B) \cong \Pc(A') \times \Pc(B) \qquad  \text{and} \qquad \Pc(A \cup B') \cong \Pc(A) \times
\Pc(B').$$
\end{itemize}
\end{lemma}

We shall use the following related result. 
\begin{lemma}\label{lem:triangles}
There exists a constant $L$, depending only on $(W,S)$, such that the following property holds. Let $\alpha_0
\subsetneq \alpha_1 \subsetneq \dots \alpha_k$ be a nested sequence of half-spaces and $m, m'$ be walls such
that $\varnothing \neq m \cap m' \subset \partial \alpha_0$, and that both $m$ and $m'$ meets $\partial
\alpha_i$ for each $i$. If $k \geq L$, then $\la r_m, r_{m'}, r_{\alpha_i} \; | \; i=0, \dots, k\ra$ is a
Euclidean triangle group and $\la r_{\alpha_i} \; | \; i=0, \dots, k\ra$ is infinite dihedral.
\end{lemma}
\begin{proof}
See \cite[Th.~8]{MR2263057}.
\end{proof}

\begin{proof}[Proof of Lemma~\ref{lemme grid lemma}]
We let $N = \max \{8, L\}$ where $L$ is the constant appearing in Lemma~\ref{lem:triangles}.

\medskip
Assume first that for some $i \in \{0, 1, \dots, k\}$ and some $j \in \{N, N+1, \dots, l-N\}$, the reflections
$r_{\alpha_i}$ and $r_{\beta_j}$ do not centralize one another. Let $\phi = r_{\alpha_i}(\beta_j)$; thus $\phi
\not \in \{\pm \alpha_i, \pm \beta_j\}$. Let $x_0 \in \partial \alpha_0 \cap \partial \beta_j$ and $x_k \in
\partial \alpha_k \cap \partial \beta_j$. Then the geodesic segment $[x_0, x_k]$ lies entirely in $\partial
\beta_j$ and crosses $\partial \alpha_i$. Since $\partial \alpha_i \cap \partial \beta_j$ is contained in
$\partial \phi$, it follows that $[x_0, x_k]$ meets $\partial \phi$. This shows that the wall $\partial \phi$
separates $x_0$ from $x_k$.

Let now $p_0 \in \partial \alpha_0 \cap \partial \beta_0$ and $p_k \in \partial \alpha_k \cap \partial \beta_0$.
Then the piecewise geodesic path $[x_0, p_0] \cup [p_0, p_k] \cup [p_k, x_k]$ is a continuous path joining $x_0$
to $x_k$. This path must therefore cross $\partial \phi$. Thus $\partial \phi$ meets either $\partial \alpha_0$
or $\partial \beta_0$ or $\partial \alpha_k$. We now deal with the case where $\partial \phi$ meets $\partial
\alpha_0$. The other two cases may be treated with analogous arguments; the straightforward adaption
will be omitted here. %???

Then $\partial \phi$ meets $\partial \alpha_m$ for each $m = 0, 1, \dots, i$. Therefore
Lemma~\ref{lem:triangles} may be applied, thereby showing that $A_i = \la r_{\alpha_m} \; | \; m=0, \dots, i\ra$
is infinite dihedral and that the subgroup $T = \la r_{\alpha_m}, r_{\beta_j}\; | \; m=0, \dots, i\ra$ is a Euclidean
triangle group. Furthermore Lemma~\ref{lem:PC}(iii) shows that $\Pc(T)$ is of irreducible affine type. Since
$\Pc(A_i)$ is infinite (because $A_i$ is infinite) and contained in $\Pc(T)$ (because $A_i$ is contained in
$T$), it follows that $\Pc(A_i) = \Pc(T)$ since any proper parabolic subgroup of $\Pc(T)$ is finite. We set $P
:= \Pc(A_i) = \Pc(T)$.

Let now $n \in \{0, 1, \dots, l\}$ with $n \neq j$. Then $r_{\beta_n}$ does not centralize $r_{\beta_i}$; in
particular it does not centralize $T$. On the other hand the wall $\partial \beta_n$ meets $\partial \alpha_m$
for all $m=0, \dots, i$, which implies by Lemma~\ref{lem:8walls} that $\la A_i \cup \{r_{\beta_n}\}\ra$ is a
Euclidean triangle group. Therefore $r_{\beta_n} \in P$ by Lemma~\ref{lem:PC}(iii).

We have already seen that $P$ is of irreducible affine type. We have just shown that $B $ is contained in $
\Pc(A_i) = P$; in particular this shows that $B$ is infinite dihedral since the walls $\partial \beta_0, \dots
\partial \beta_l$ are pairwise parallel. Moreover, the group $\la B \cup \{r_{\alpha_i}\} \ra$ must be a
Euclidean triangle group since it is a subgroup of $P$. In particular we have $\Pc(B) = P$ by
Lemma~\ref{lem:PC}(iii). Since every $\partial \alpha_m$ meets every $\partial \beta_j$, the same arguments as
before now show that $r_{\alpha_m} \in \Pc(B) = P$ for all $m = i+1, \dots, k$. Finally we conclude that $\Pc(A)
= \Pc(B) = P$ in this case.

\medskip
Notice that, in view of the symmetry between the $\alpha$'s and the $\beta$'s, the previous arguments yield the
same conclusion if one assumed instead that for some $i \in \{N, N+1, \dots, k-N\}$ and some $j \in \{0, 1,
\dots l\}$, the reflections $r_{\alpha_i}$ and $r_{\beta_j}$ do not centralize one another.

\medskip
Assume now that for all  $i \in \{0, \dots, k\}$ and all  $j \in \{N, N+1, \dots, l-N\}$, the reflections
$r_{\alpha_i}$ and $r_{\beta_j}$ commute and that, furthermore, for all $i \in \{N, N+1, \dots, k-N\}$ and all
$j \in \{0, 1, \dots l\}$, the reflections $r_{\alpha_i}$ and $r_{\beta_j}$ commute. By Lemma~\ref{lem:PC}(ii)
the parabolic closures $\Pc(A)$, $\Pc(A')$ $\Pc(B)$ and $\Pc(B')$ are of irreducible type. By assumption $A'$
centralizes $B$. By Lemma~\ref{lem:PC}(iv), either $\Pc(A') = \Pc(B)$ is of affine type or else $\Pc(A' \cup B)
\cong \Pc(A') \times \Pc(B)$. In the former case, we may argue as before to conclude again that $\Pc(A) =
\Pc(B)$ is of affine type and we are in case (i) of the alternative. Otherwise, we have $\Pc(A' \cup B) \cong
\Pc(A') \times \Pc(B)$ and by similar arguments we deduce that $\Pc(A \cup B') \cong \Pc(A) \times \Pc(B')$.
\end{proof}

%%%%%%%%%%%%%%%%%%%%%%%%%%%%%%%%%%%%%%%%%%%%%%%%%%%%%%%%%%
\subsection{Orbits of essential roots: affine versus non-affine}
%%%%%%%%%%%%%%%%%%%%%%%%%%%%%%%%%%%%%%%%%%%%%%%%%%%%%%%%%

Using the Grid Lemma, we can now establish a basic description of the $w$-orbit of a $w$-essential wall for some fixed $w \in W$. As before, we let $W_0 < W$ be a torsion-free finite index normal subgroup. Recall from Lemma~\ref{lem:essential=transverse} that for all $n>0$ we have $\Ess(w) = \Ess(w^n)$ and, moreover, the set $\Ess(w)$ has finitely many orbits under the action of $\la w \ra$ (and hence also under $\la w^n \ra$). 

\begin{prop}\label{prop:Orbits}
Let $w \in W$ be of infinite order, let $k>0$ be such that $w^k \in W_0$ and let $\Ess(w) = \Ess(w^k) = M_1 \cup \dots \cup M_t$ be the partition of $\Ess(w)$ into $\la w^k \ra$-orbits. For each $i \in \{1, \dots, t\}$, let also  $P_i =  \Pc(\{r_m \; | \; m \in M_i  \})$. 

Then for all $i \in \{1, \dots, t\}$, the group $P_i$ is an irreducible direct component of $\Pc(w)$. In particular, for all $j \neq i$, we have either $P_i = P_j$ or  $\Pc(P_i \cup P_j) \cong P_i \times P_j$. More precisely, one of the following assertions holds.
\begin{enumerate}[(i)]
\item  $P_i = P_j$ and each $m \in M_i$ meets finitely many walls in $M_j$. 

\item $P_i = P_j$ is irreducible affine. 

\item $\Pc(P_i \cup P_j) \cong P_i \times P_j$.
\end{enumerate}

\end{prop}

\begin{proof}
Let $i \in \{1, \dots, t\}$. Since $M_i$ is $\la w^k \ra$-invariant, it follows that $P_i$ is normalized by $w^k$. As $\la r_m \; | \; m \in M_i  \ra$ is an irreducible reflection group by Lemma~\ref{lem:PC}(ii), $P_i$ is of irreducible non-spherical type by Lemma~\ref{lem:PC}(iii). It then follows from Lemma~\ref{lemme Deodhar} that $\mathscr N(P_i) = P_i \times \mathscr Z(P_i)$ is itself a parabolic subgroup. In particular it contains $\Pc(w^k)$. Since on the other hand we have $P_i \leq \Pc(w^k)$ by Lemma~\ref{lemme Krammer}, we infer that $P_i$ is a direct component  of $\Pc(w^k)$. Since $\Pc(w^k) = \Pc^\infty(w)$ is the essential component of $\Pc(w)$ by Lemma~\ref{lemme Krammer}, we deduce that $P_i$ is a direct component of $\Pc(w)$ as desired.

\medskip
Let now $j \neq i$. Since we already know that $P_i$ and $P_j $ are irreducible direct components of $\Pc(w)$, it follows that either $P_i = P_j$ or (iii) holds. 
So assume that $P_i=P_j$ and that there exists a wall $m \in M_i$ meeting infinitely many walls in $M_j$. We have to show that (ii) holds. 

Let $\lambda$ be a $w$-axis. By Lemma~\ref{lem:essential=transverse}, all walls in $M_i \cup M_j$ are transverse to $\lambda$. Moreover, by Lemma~\ref{lemme DJ} the elements of $M_i$ (resp. $M_j$) are pairwise parallel. Therefore, we deduce that infinitely many walls in $M_i$ meet infinitely many walls in $M_j$. Since $M_i$ and $M_j$ are both  $\la w^k \ra$-invariant, it follows that all walls in $M_i$ meet all walls in $M_j$. Thus $M_i \cup M_j$ forms a grid and the desired conclusion follows from Lemma~\ref{lemme grid lemma}.
\end{proof}

We shall now deduce a rather subtle, but nevertheless important, difference between the affine and non-affine cases concerning the $\la w \ra$-orbit of a $w$-essential root $\alpha$.

Let us start by considering a specific example, namely  the Coxeter group $W = \la r_a, r_b, r_c \ra$ of type $\tilde A_2$, acting on the Euclidean plane. One verifies easily that $W$ contains a nonzero translation $t$ which preserves the $r_a$-invariant wall $m_a$. Let $w = t r_a$. Then $w$ is of infinite order so that $\Pc(w) = W$. Moreover the walls $m_b$ and $m_c$, respectively fixed by $r_b$ and $r_c$, are both $w$-essential by Lemma~\ref{lem:essential=transverse}. Now we observe that, for each even integer $n$ the walls $m_b$ and 
$w^nm_b$ are parallel, while for each odd integer the walls $m_b$ and 
$w^nm_b$ have a non-empty intersection.

The following result (in the special case $m = m'$) shows that the situation we have just described cannot occur in the non-affine case. 

\begin{prop}\label{prop:affVSnonaff}
Let $w \in W$, $m$ be a $w$-essential wall and $P$ be the irreducible component of $\Pc(w)$ that contains $r_m$. 

If $P$ is not of affine type, then for each $w$-essential wall $m'$ such that $r_{m'} \in P$, there exists an $l_0\in\NN$  such that for all $l\in\ZZ$ with $|l| \geq l_0$, the wall $m'$ lies between $w^{-l}m$ and $w^lm$. 
\end{prop}

\begin{proof}
First notice that if $m$ is a $w$-essential wall, then the reflection $r_m$ belongs to $\Pc(w)$ by Lemma~\ref{lemme Krammer}, so that $P$ is well defined. Moreover, we have $r_{w^lm} = w^l r_m w^{-l} \in P$ for all $l \in \ZZ$. 

Let $k>0$ be such that $w^k \in W_0$ and let $\Ess(w) = \Ess(w^k) = M_1 \cup \dots \cup M_t$ be the partition of $\Ess(w)$ into $\la w^k \ra$-orbits. Upon reordering the $M_i$, we may assume that $m' \in M_1$. Let also $I \subseteq \{1, \dots, t\}$ be the set of those $i$ such that $w^lm \in M_i$ for some $l$. In other words the $\la w\ra $-orbit of $m$ coincides with $\bigcup_{i \in I} M_i$. 

 For all $j$, set $P_j =  \Pc(\{r_{\mu} \; | \; \mu \in M_j  \})$. By Proposition~\ref{prop:Orbits}, each $P_j$ is an irreducible direct component of $\Pc(w)$. By hypothesis, this implies that  $P = P_1 =P_i$  for all $i \in I$.

Suppose now that for infinitely many values of $l$, the wall $w^l m$ has a non-empty intersection with $m'$. We have to deduce that $P$ is of affine type. 

Recall from Lemma~\ref{lemme DJ} that the elements of $M_j$ are pairwise parallel for all $j$. Therefore, our assumption  implies that for some $i \in I$, the wall $m'$ meets infinitely many walls in $M_i$. By Proposition~\ref{prop:Orbits}, this implies that either $P = P_1  = P_i$ is of affine type, or $\Pc(P_1 \cup P_i) \cong P_1 \times P_i$.  The second case is impossible since $P_1 = P_i$. 
\end{proof}

%%%%%%%%%%%%%%%%%%%%%%%%%%%%%%%%%%%%%%%%%%%%%%%%%%%%%%%%%%
\subsection{On parabolic closures of a pair of reflections}
%%%%%%%%%%%%%%%%%%%%%%%%%%%%%%%%%%%%%%%%%%%%%%%%%%%%%%%%%%

The following consequence of Proposition~\ref{prop:Orbits} was stated as Theorem~\ref{thm:Cox1} in the introduction. 

\begin{cor}\label{cor:Pc:2refl}
For each $w \in W$ with infinite irreducible parabolic closure $\Pc(w)$,  there is a constant $C$ such that the following holds. For all $m, m' \in \Ess(w)$ with 
$d(m, m')>C$, we have $\Pc(w) =   \Pc(r_{m}, r_{m'}) $.
\end{cor}

We shall use the following.

\begin{lemma}\label{lemme 3 paralleles}
Let $\alpha,\beta,\gamma\in\Phi$ such that $\alpha\subsetneq\beta\subsetneq\gamma$. Then $r_{\beta}\in\Pc(\{r_{\alpha},r_{\gamma}\})$. 
\end{lemma}
\begin{proof}
See \cite[Lemma 17]{MR2263057}.
\end{proof}

\begin{proof}[Proof of Corollary~\ref{cor:Pc:2refl}]
Retain the notation of Proposition~\ref{prop:Orbits}. Since $P = \Pc(w)$ is irreducible, we have $P = P_i$ for all $i \in \{1, \dots, t\}$ by Proposition~\ref{prop:Orbits}. Recall that $M_i$ is the $\la w^k\ra$-orbit of some $w$-essential wall $m$. For all $n \in \ZZ$, we set $m_n = w^{kn}m$. By Lemma~\ref{lemme DJ} the elements of $M_i$ are pairwise parallel and hence for all $i<j<n$, it follows that $m_j$ separates $m_i$ from $m_n$. For all $n\geq 0$ let now $Q_n = \Pc(\{r_{m_n}, r_{m_{-n}}\})$. By  Lemma~\ref{lemme 3 paralleles} we have $Q_n \leq Q_{n+1} \leq P$ for all $n \geq 0$. In particular $\bigcup_{n \geq 0} Q_n$ is a parabolic subgroup, which must thus coincide with $P$. It follows that $Q_n = P$ for some $n$. Since this argument holds for all $i \in \{1, \dots, t\}$, the desired result follows.
\end{proof}

\begin{cor}
Any irreducible non-spherical parabolic subgroup $P$ is the parabolic closure of a pair of reflections.
\end{cor}
\begin{proof}
Let $w\in P$ such that $P=\Pc(w)$. Such an $w$ always exists by \cite[Cor.4.3]{MR2585575}. (Note that this can also be deduced from Corollary~\ref{cor thm Coxeter} below together with \cite[Prop.2.43]{ABrown}.) The conclusion now follows from Corollary~\ref{cor:Pc:2refl}.
\end{proof}

%%%%%%%%%%%%%%%%%%%%%%%%%%%%%%%%%%%%%%%%%%%%%%%%%%%%%%%%%%%%%%%%%%%%%%%%%%%%%%%
\subsection{The parabolic closure of a product of two elements in a Coxeter group}
%%%%%%%%%%%%%%%%%%%%%%%%%%%%%%%%%%%%%%%%%%%%%%%%%%%%%%%%%%%%%%%%%%%%%%%%%%%%%%%

We are now able to present the main result of this section, which was stated as Theorem~\ref{thm:Cox2} in the introduction.

Before we state it, we prove one more technical lemma about $\CAT$ spaces. Recall that $W$ acts on the $\CAT$ space $X$. For a hyperbolic $w\in W$, let $|w|$ denote its translation length and set $\Min(w)=\{x\in X \ | \ \dist(x,wx)=|w|\}$.

\begin{lemma}\label{lemme cat0}
Let $w\in W$ be hyperbolic and suppose it decomposes as a product $w=w_1w_2\dots w_t$ of pairwise commuting hyperbolic elements of $W$. Let $m$ be a $w$-essential wall. Then $m$ is also $w_i$-essential for some $i\in\{1,\dots,t\}$.
\end{lemma}
\begin{proof}
Write $w_0:=w$. Then, since the $w_i$ are pairwise commuting for $i=0,\dots,t$, each $w_i$ stabilizes $\Min(w_j)$ for all $j$. Thus $M:=\bigcap_{j=1}^t{\Min(w_j)}$ and $\Min(w)$ are both non-empty by $\CAT$-convexity, and are stabilized by each $w_i$, $i=0,\dots,t$. Therefore, if $x\in M\cap\Min(g)$, there is a piecewise geodesic path $x,w_1x,w_1w_2x,\dots,w_1\dots w_tx=wx$ inside $M\cap\Min(g)$, where each geodesic segment is part of a $w_i$-axis for some $i\in\{1,\dots,t\}$. Since any wall intersecting the geodesic segment $[x,wx]$ must intersect one of those axis, the conclusion follows from Lemma~\ref{lem:essential=transverse}.
\end{proof}

\begin{thm}\label{thm thm Coxeter}
For all $g,h\in W_0$, there exists a constant $K=K(g,h)\in\NN$ such that for all $m,n\in\ZZ$ 
with $\min\{|m|,|n|,|m/n|+|n/m|\}\geq K$, 
we have $\Pc(g)\cup\Pc(h)\subseteq\Pc(g^mh^n)$. 
\end{thm}
\begin{proof}

Fix $g,h\in W_0$. Let $\Ess(g) = M_1 \cup \dots \cup M_k$ (resp. $\Ess(h) = N_1 \cup \dots \cup N_l$) be the partition of $\Ess(g)$ into $\la g \ra$-orbits (resp. $\Ess(h)$ into $\la h\ra$-orbits). For all $i \in \{1,\dots,k\}$ and $j \in \{1,\dots,l\}$, set $P_i = \Pc(\{r_m \; | \; m \in M_i\})$ and $Q_j = \Pc(\{r_m \; | \; m \in N_j\})$. 

By Lemma~\ref{lemme Krammer}, we have $\Pc(g) = \la \{r_m \; | \; m \in M_i,  \ i =1,\dots,k\} \ra$, and Proposition~\ref{prop:Orbits} ensures that $P_i$ is an irreducible direct component of $\Pc(g)$ for all $i$. Thus there is a subset $I \subseteq \{1,\dots,k\}$ such that $\Pc(g) = \prod_{i \in I} P_i$. Similarly, there is a subset $J \subseteq \{1,\dots,l\}$ such that $\Pc(h) = \prod_{j \in J} Q_j$. 

For all $i \in I$ and $j \in J$, we finally let $g_i$ and $h_j$ denote the respective projections of $g$ and $h$ onto $P_i$ and $Q_j$, so that $P_i=\Pc(g_i)$ and  $Q_j=\Pc(h_j)$.

We define a collection $\E(g,h)$ of subsets of $W$ as follows: a set $Z\subseteq W$ belongs to 
$\E(g, h)$ if and only if there exists a constant $K=K(g,h,Z)\in\NN$ such that for all $m,n\in\ZZ$ with $\min\{|m|,|n|,|m/n|+|n/m|\}\geq K$ we have $Z\subseteq\Pc(g^mh^n)$. 

\medskip
Our goal is to prove that $\Pc(g)$ and $\Pc(h)$ both belong to $E(g, h)$. 
To this end,   it   suffices to show that $P_i$ and  $Q_j$ belong to $\E(g,h)$ for all $i \in I $ and $j \in J$.  This will be achieved in Claim~\ref{claim 5} below.

\begin{claim}\label{claim 0}
$M_s\subseteq \Ess(g_i)$ for all $s\in\{1,\dots,k\}$ and $i\in I$ such that $P_s=P_i$.
Similarly, $N_s\subseteq\Ess(h_j)$ for all $s\in\{1,\dots,l\}$ and $j\in J$ such that $Q_s=Q_j$.
\end{claim}

Indeed, let $m\in M_s$ for some $s\in\{1,\dots,k\}$. Then $r_m\in P_s=P_i$. Moreover, as $m$ is $g$-essential, it must be $g_{i'}$-essential for some $i'\in I$ by Lemma~\ref{lemme cat0}. But then $r_m\in\Pc(g_{i'})=P_{i'}$ and so $i'=i$. The proof of the second statement is similar.

\begin{claim}\label{claim 1}
If $i \in I$ is such that $[P_i,Q_j]=1$ for all $j\in J$, then $P_i$ belongs to $\E(g,h)$. 

Similarly, if $j \in J$ is such that $[P_i,Q_j]=1$ for  all $i\in I$, then $Q_j$ belongs to $\E(g,h)$.
\end{claim}
Indeed, suppose $[P_i,Q_j]=1$ for some $i\in I$ and for all $j\in J$. Then $P_i$ commutes with $\Pc(h)$. Thus $h$ fixes every wall of $M_i$. In particular, any wall $\mu \in M_i$ is $g^mh^n$-essential for all $m,n\in\ZZ^*$ since $g^mh^n=g_i^mw$ for some $w\in W$ fixing $\mu$ and commuting with $g_i$. Therefore $P_i\subseteq \Pc(g^mh^n)$ for all $m,n\in\ZZ^*$ and so $P_i$ belongs to $\E(g,h)$. The second statement is proven in the same way.

\begin{claim}\label{claim 2}
Let  $i\in I$ and $j\in J$ be such that  $P_i=Q_j$. Then, for all $m,n\in\ZZ$, every $g_i^mh_j^n$-essential root is also $g^mh^n$-essential. 
\end{claim}
Indeed, take $\alpha\in\Phi$ and $k>0$ such that $(g_i^mh_j^n)^k\alpha\subsetneq\alpha$. Notice that $\Pc(g_i^mh_j^n)\subseteq P_i=Q_j$, and hence $r_\alpha \in P_i = Q_j$ by Lemma~\ref{lemme Krammer}. Moreover, setting $g':=\prod_{t\neq i}{g_t^m}$ and $h':=\prod_{t\neq j}{h_t^n}$, we have $g'\alpha=\alpha=h'\alpha$ since $g'$ and $h'$ centralize $P_i=Q_j$. Therefore $(g^mh^n)^k\alpha =  (g_i^mh_j^n)^k(g'h')^k\alpha= (g_i^mh_j^n)^k\alpha\subsetneq \alpha$  so that $\alpha$ is also $g^mh^n$-essential.

\begin{claim}\label{claim 3}
Let  $i\in I$ and $j\in J$ be such that  $P_i=Q_j$. If $P_i$ is of affine type, then $P_i = Q_j$ belongs to $\E(g,h)$.
\end{claim}

Since $P_i = Q_j$ is of irreducible affine type, we have $\Pc(w) = P_i$  for all $w \in P_i$ of infinite order. Thus, in order to prove the claim, it suffices to show that there exists some constant $K$ such that $g_i^m h_j^n$ is of infinite order for all $m, n \in\ZZ$ with $\min\{|m|,|n|,|m/n|+|n/m|\}\geq K$. Indeed, we will then get that $\Pc(g_i^mh_j^n)=P_i$ is of essential type and so $P_i=\Pc(g_i^mh_j^n) \leq \Pc(g^m h^n)$ by Claim~\ref{claim 2} and Lemma~\ref{lemme Krammer}(ii).

 Recalling that $P_i$ is of affine type, we can argue in the geometric realization of a Coxeter complex of affine type, which is a Euclidean space. We deduce that if $g_i$ and $h_j$ have non-parallel translation axes, then $g_i^m h_j^n$ is of infinite order for all nonzero $m, n$. On the other hand,  if $   g_i$ and $h_j$ have some parallel translation axes, we consider a Euclidean hyperplane $H$ orthogonal to these and let $\ell_i$ and $\ell_j$ denote the respective translation lengths of $g_i$ and $h_j$. Then, upon replacing $g_i$ by its inverse (which does not affect the conclusion since $E(g, h) = E(g\inv, h)$), we have $d(g_i^m h_j^n H, H) = | m \ell_i - n \ell_j |$. Since $g_i^mh_j^n$ is of infinite order as soon as this distance is nonzero, the claim now follows by setting  $K = \ell_i / \ell_j+\ell_j/\ell_i+1$.

\begin{claim}\label{claim 4}
Let  $i\in \{1,\dots,k\}$ and $j\in \{1,\dots,l\}$ be such that $M_i\cap N_j$ is infinite. Then $P_i=Q_j$ and these belong to $\E(g,h)$. 
\end{claim} 

Indeed, remember that the walls in $M_i$ are pairwise parallel by Lemma~\ref{lemme DJ}. Since $M_i\cap N_j\subseteq \Ess(g_{i'})\cap\Ess(h_{j'})$ for some $i'\in I$ such that $P_i=P_{i'}$ and some $j'\in J$ such that $Q_j=Q_{j'}$ by Claim~\ref{claim 0}, Corollary~\ref{cor:Pc:2refl} then yields $P_i=Q_j$.

Let now $C$ denote the minimal distance between two parallel walls in $X$ and set $K:=\frac{|g|+|h|}{C}+1$. Let $m,n\in\ZZ$ be such that $\min\{|m|,|n|,|m/n|+|n/m|\}\geq K$. We now show that $P_i\leq\Pc(g^mh^n)$. By Lemma~\ref{lemme Krammer} and Corollary~\ref{cor:Pc:2refl}, it is sufficient to check that infinitely many walls in $M_i\cap N_j$ are $g^mh^n$-essential. 

Note first that for any wall $\mu\in M_i\cap N_j$, we have $g^{\epsilon m}\mu\in M_i$ and $h^{\epsilon n}\mu\in N_j$ for $\epsilon\in\{+,-\}$. Thus, since $M_i\cap N_j$ is infinite, there exist infinitely many such $\mu\in M_i\cap N_j$ with the property that $g^{\epsilon m}\mu$ lies between $\mu$ and some $\mu_{\epsilon}\in M_i\cap N_j$ and $h^{\epsilon n}\mu$ lies between $\mu$ and some $\mu'_{\epsilon}\in M_i\cap N_j$ for $\epsilon\in\{+,-\}$. We now show that any such $\mu$ is $g^mh^n$-essential, as desired. Consider thus such a $\mu$. 

Let $D$ be a $g$-axis and $D'$ be an $h$-axis.
Since $M_i\cap N_j\subseteq \Ess(g)\cap\Ess(h)$, Lemma~\ref{lem:essential=transverse} implies that each of the walls $\mu$, $\mu_{\epsilon}$ and $\mu'_{\epsilon}$ for $\epsilon\in\{+,-\}$ is transverse to both $D$ and $D'$. In particular, the choice of $\mu$ implies that $g^{\epsilon m}\mu$ and $h^{\epsilon n}\mu$ for $\epsilon\in\{+,-\}$ are also transverse to both $D$ and $D'$.

Let $\alpha\in\Phi$ be such that $\partial\alpha=\mu$ and $g^m\alpha\subsetneq\alpha$. 
If $h^n\alpha\subsetneq\alpha$ then clearly $g^mh^n\alpha\subsetneq\alpha$, as desired. Suppose now that $h^n\alpha\supsetneq\alpha$.

Note that the walls in $\la g\ra \mu\cup \la h\ra\mu$ are pairwise parallel since this is the case for the walls in $W_0\cdot\mu$ by Lemma~\ref{lemme DJ} and since $g,h\in W_0$.

Assume now that $|n|>|m|$, the other case being similar. In particular, $|n/m|>|g|/C$. Then $\dist(\mu,g^{-m}\mu)\leq |m|\cdot |g|< |n|\cdot C\leq \dist(\mu,h^n\mu)$ and so the wall $g^{-m}\mu$ lies between $\mu$ and $h^n\mu$. Thus $\alpha\subsetneq g^{-m}\alpha\subsetneq h^n\alpha$ and so $g^mh^n\alpha\supsetneq \alpha$, as desired.

\begin{claim}\label{claim 5}
For all  $i\in I$ and $j\in J$, the sets $P_i $ and $ Q_j$ both belong to $\E(g,h)$.
\end{claim} 
We only deal with $P_i$; the argument for $Q_j$ is similar. 

Let $D$ denote a $g$-axis, and $D'$ an $h$-axis in $X$. By Claim~\ref{claim 4} we may assume that $M_i \cap \Ess(h)$ is finite. Moreover, by Claim~\ref{claim 2} we may assume there exists a $j\in J$ such that $[P_i,Q_j]\neq 1$. 

If $N_j\cap \Ess(g)$ is infinite, then $N_j\cap M_{i'}$ is infinite for some $i'\in\{1,\dots,k\}$ and thus Claim~\ref{claim 4} yields that $Q_j=P_{i'}\in\E(g,h)$. In particular, $[Q_j,P_s]=1$ as soon as $P_s\neq P_{i'}$. This implies $P_i=P_{i'}\in \E(g,h)$, as desired. We now assume that $N_j\cap \Ess(g)$ is finite.

Thus by Lemma~\ref{lem:essential=transverse}, only finitely many walls in $M_i$ intersect $D'$ and only finitely walls in $N_j$ intersect $D$.

Take $m_1\in M_i$ and $m_2\in N_j$. By Claim~\ref{claim 0} and Corollary~\ref{cor:Pc:2refl}, there exists some $k_0\in\NN$ such that if one sets $M:=\{g^{sk_0}m_1 \ | \ s\in\ZZ\}\subseteq M_{i}$ and $N:=\{h^{tk_0}m_2 \ | \ t\in\ZZ\}\subseteq N_{j}$, then any two reflections associated to distinct walls of $M$ (respectively, $N$) generate $P_{i}$ (respectively, $Q_{j}$) as parabolic subgroups. Also, we may assume that no wall in $M$ intersects $D'$ and that no wall in $N$ intersects $D$. 

If every wall of $M$ intersects every wall of $N$, then since $[P_i,Q_j]\neq 1$, Lemma~\ref{lemme grid lemma} yields that $P_i=Q_j$ is of affine type and Claim~\ref{claim 3} allows us to conclude. Up to making a different choice for $m_1$ and $m_2$ inside $M$ and $N$ respectively, we may thus assume that $m_1$ is parallel to $m_2$. For the same reason, we may also choose $m_1'\in M$ and $m_2'\in N$ such that $D'$ lies between $m_1$ and $m_1'$, $D$ lies between $m_2$ and $m_2'$, and such that $m_1\cap m_2'=m_2\cap m_1'=m_1'\cap m_2'=\varnothing$.

Let now $s_0,t_0\in\ZZ$ be such that $g^{s_0k_0}m_1=m_1'$ and $h^{t_0k_0}m_2=m_2'$. Up to interchanging $m_1$ and $m_1'$ (respectively, $m_2$ and $m_2'$), we may assume that $s_0>0$ and $t_0>0$.

Let $\alpha,\beta\in\Phi$ be such that $\partial\alpha=m_1$, $\partial\beta=m_2$ and such that $D'$ is contained in $\alpha\cap - g^{s_0k_0}\alpha$ and $D$ is contained in $\beta\cap -h^{t_0k_0}\beta$. For each $s,t\in\ZZ$, set $\alpha_s:=g^{sk_0}\alpha$ and $\beta_t=h^{tk_0}\beta$ (see Figure~\ref{thmCox}).
Since for two roots $\gamma,\delta\in\Phi$ with $\partial\gamma$ parallel to $\partial\delta$, one of the possibilities $\gamma\subseteq\delta$ or $\gamma\subseteq -\delta$ or $-\gamma\subseteq\delta$ or $-\gamma\subseteq -\delta$ must hold, this implies that 
$$\alpha_{s_0}\subseteq -\beta_{t_0}, \quad -\alpha\subseteq \beta \quad\textrm{and}\quad \beta_{t_0}\subseteq\alpha.$$ 

\begin{figure}
\begin{center}
\def\svgwidth{6.5cm}
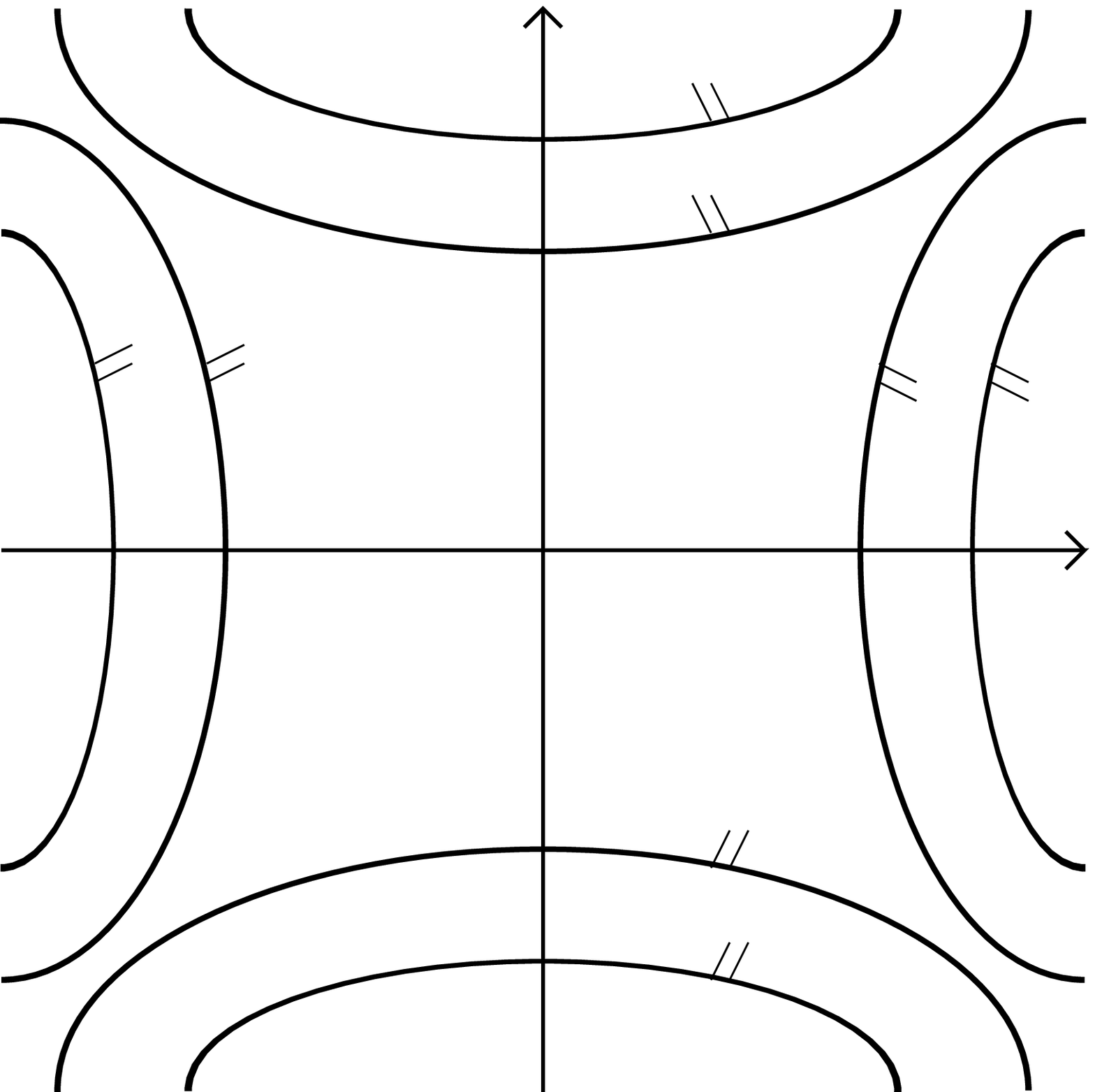
\caption{Claim \ref{claim 5}.}
\label{thmCox}
\end{center}
\end{figure}

Set $K:=(s_0+t_0+1)k_0$ and let $m,n\in\ZZ$ be such that $|m|,|n|>K$. We now prove that $P_i\leq \Pc(g^mh^n)$. By Lemma~\ref{lemme Krammer}, it is sufficient to show that either $\alpha_{-1}$ and $\alpha$ or $\alpha_{s_0}$ and $\alpha_{s_0+1}$ are $g^mh^n$-essential.
We distinguish several cases depending on the respective signs of $m,n$.

\begin{itemize}
\item
If $m,n>0$, then
$$g^mh^n\alpha_{s_0+1}\subseteq g^mh^n\alpha_{s_0}\subseteq g^mh^n\beta\subsetneq g^m\beta_{t_0}\subseteq g^m\alpha\subsetneq \alpha_{s_0+1}\subseteq\alpha_{s_0}$$
so that $\alpha_{s_0}$ and $\alpha_{s_0+1}$ are $g^mh^n$-essential. 

\item
If $m,n<0$, then
$$g^mh^n\alpha_{-1}\supseteq g^mh^n\alpha\supseteq g^mh^n\beta_{t_0}\supsetneq g^m\beta\supseteq g^m\alpha_{s_0}\supsetneq \alpha_{-1}\supseteq \alpha$$
so that $\alpha_{-1}$ and $\alpha$ are $g^mh^n$-essential. 

\item
If $m>0$ and $n<0$, then
$$g^mh^n\alpha_{s_0+1}\subseteq g^mh^n\alpha_{s_0}\subseteq g^mh^n(-\beta_{t_0})\subsetneq g^m(-\beta)\subseteq g^m\alpha\subsetneq \alpha_{s_0+1}\subseteq\alpha_{s_0}$$
so that $\alpha_{s_0}$ and $\alpha_{s_0+1}$ are $g^mh^n$-essential. 

\item
If $m<0$ and $n>0$, then
$$g^mh^n\alpha_{-1}\supseteq g^mh^n\alpha\supseteq g^mh^n(-\beta)\supsetneq g^m(-\beta_{t_0})\supseteq g^m\alpha_{s_0}\supsetneq \alpha_{-1}\supseteq \alpha$$
so that $\alpha_{-1}$ and $\alpha$ are $g^mh^n$-essential. 
\end{itemize}

This concludes the proof of the theorem.
\end{proof}

The following corollary will be of fundamental importance in the rest of the paper. It was stated as Corollary~\ref{cor:fundamental} in the introduction. 

\begin{cor}\label{cor thm Coxeter}
Let $H$ be a subgroup of $W$. Then there exists $h\in H\cap W_0$ such that $[\Pc(H):\Pc(h)]<\infty$.
\end{cor}
\begin{proof}
Take $h\in H\cap W_0$ such that $\Pc(h)$ is maximal. Then $\Pc(h)=\Pc(H\cap W_0)$, for otherwise there would exist $g\in H\cap W_0$ such that $\Pc(g)\not\subseteq \Pc(h)$, and hence Theorem~\ref{thm thm Coxeter} would yield integers $m,n$ such that $\Pc(h)\subsetneq \Pc(g^mh^n)$, contradicting the choice of $h$. The result now follows from Lemma~\ref{lemme parabolique indice fini} since $[H:H\cap W_0]<\infty$.
\end{proof}

\begin{rem}
Note that the conclusion of Corollary~\ref{cor thm Coxeter} cannot be improved: indeed, one cannot expect that there is some $h \in H$ such that $\Pc(H) = \Pc(h)$ in general. Consider for example the Coxeter group $W=\la s\ra\times\la t\ra\times\la u\ra$, which is a direct product of three copies of $\ZZ/2\ZZ$. Then the parabolic closure of the subgroup $H=\la st,tu\ra$ of $W$ is the whole of $W$, but there is no $h\in H$ such that $\Pc(h)=W$.
\end{rem}

%%%%%%%%%%%%%%%%%%%%%%%%%%%%%%%%%%%%%%%%%%%%%%%%%%%%%%%%%%%
\subsection{On walls at bounded distance from a residue}
%%%%%%%%%%%%%%%%%%%%%%%%%%%%%%%%%%%%%%%%%%%%%%%%%%%%%%%%%%%

We finish this section with a couple of observations on Coxeter groups which we shall need in our study of open subgroups of Kac--Moody groups.

Given a subset $J \subseteq S$, we set $\Phi_J=\{\alpha\in\Phi \ | \ \exists v\in W_J, \ s\in J: \alpha=v\alpha_s\}$, where $\alpha_s$ denotes the positive root associated with the reflection $s$.

\begin{lemma}\label{lemme racine essentielle}
Let $L\subseteq S$ be essential. Then for each root $\alpha\in\Phi_L$, there exists $w\in W_L$ such that $w.\alpha \subsetneq \alpha$. In particular $\alpha$ is $w$-essential.
\end{lemma}
\begin{proof}
Let $\alpha\in\Phi_L$. By \cite[Prop.\thinspace 8.1, p.\thinspace 309]{Hee}, there exists a root $\beta\in\Phi_L$ such that $\alpha\cap\beta=\varnothing$. We can then take $w=r_{\alpha}r_{\beta}$ or its inverse. 
\end{proof}

\begin{lemma}\label{lemme 8 murs}
Let $L\subseteq S$ be essential, and let $R$ be the standard $L$-residue of the Coxeter complex $\Sigma$ of $W$. 

Then for each wall $m$ of $\Sigma$, the following assertions are equivalent:
\begin{itemize}
\item[(i)]
$m$ is perpendicular to every wall of $R$,
\item[(ii)]
$[r_{m},W_L]=1$,
\item[(iii)]
There exists $n>0$ such that $R$ is contained in an $n$-neighbourhood of $m$.
\end{itemize} 
\end{lemma}
\begin{proof}
We first show that (iii)$\Rightarrow$(ii). By Lemma~\ref{lemme racine essentielle}, if $m'$ is a wall of $R$ (that is, a wall intersecting $R$), then there exists $w\in W_L$ such that one of the two half-spaces associated to $m'$ is $w$-essential. It follows that $m$ and $m'$ cannot be parallel since $R$ is at a bounded distance from $m$. Hence $m$ is transversal to every wall of $R$, and does not intersect $R$. Back to an arbitrary wall $m'$ of $R$, consider a wall $m''$ of $R$ that is parallel to $m'$ and such that the reflection group generated by the two reflections $r_{m'}$ and $r_{m''}$ is infinite dihedral. Such a wall $m''$ exists by Lemma~\ref{lemme racine essentielle}. Then $r_m$ centralizes these reflections by Lemma~\ref{lem:8walls} and \cite[Lem.12]{CaRe}. As $m'$ was arbitrary, this means that $r_m$ centralizes $W_L$.

The equivalence of (i) and (ii) is trivial.

Finally, to show (i)$\Rightarrow$(iii), notice that if $C$ is a chamber of $R$ and $t$ a reflection associated to a wall of $R$, then the distance from $C$ to $m$ equals the distance from $t\cdot C$ to $m$. Indeed, if $\alpha$ is the root associated to $m$ not containing $R$ and $D$ is the projection of $C$ onto $\alpha$, then $t\cdot D$ is the projection of $t\cdot C$ onto $\alpha$. As $W_L$ is transitive on $R$, (iii) follows.
\end{proof}

%%%%%%%%%%%%%%%%%%%%%%%%%%%%%%%%%%%%%%%%%%%%%%%%%%%%%
%%%%%%%%%%%%%%%%%%%%%%%%%%%%%%%%%%%%%%%%%%%%%%%%%%%%%
\section{Open and parabolic subgroups of Kac--Moody groups}
%%%%%%%%%%%%%%%%%%%%%%%%%%%%%%%%%%%%%%%%%%%%%%%%%%%%%
%%%%%%%%%%%%%%%%%%%%%%%%%%%%%%%%%%%%%%%%%%%%%%%%%%%%%

Basics on Kac--Moody groups  and their completions can be found in \cite{theseBR}, \cite{CaRe} and references therein. We focus here on the case of a finite ground field.

Let $\GGG=\GGG(\FF_q)$ be a (minimal) Kac--Moody group over a finite field $\FF_q$ of order $q$. The group $\GGG$ is endowed with a root group datum $\{U_{\alpha} \; | \; \alpha\in\Phi=\Phi(\Sigma(W,S))\}$ for some Coxeter system $(W,S)$, which yields a twin BN-pair $(\mathcal B_+, \mathcal B_-, \mathcal N)$ with associated twin building $(\Delta_+, \Delta_-)$. Let $C_0$ be the fundamental chamber of $\Delta_+$, namely the chamber such that $\mathcal B_+ = \Stab_\GGG(C_0)$, and let $A_0 \subset \Delta_+$ be the fundamental apartment, so that $\mathcal N = \Stab_\GGG(A_0)$ and $\mathcal H := \mathcal B_+ \cap \mathcal N = \Fix_\GGG(A_0)$. 
We  identify $\Phi$ with the set of half-spaces of  $A_0$. 

We next let $G$ be the completion of $\GGG$ with respect to the positive building topology. Thus the finitely generated group $\GGG$ embeds densely in the topological group $G$, which is locally compact, totally disconnected and acts properly and continuously on $\Delta := \Delta_+$ by automorphims. A completed Kac--Moody group over a finite field shall be called a \textbf{locally compact Kac--Moody group}. Let $B = \overline{\mathcal B_+}$ be the closure of $\mathcal B_+$ in $G$, let $N = \Stab_G(A_0)$ and $H = B \cap N = \Fix_G(A_0)$. [We warn the reader that $\mathcal N$ and $\mathcal H$ are discrete, whence closed in $G$ while $N$ and $H$ are non-discrete closed subgroups.]
The pair $(B, N)$ is a BN-pair of type $(W, S)$ for $G$; in particular we have $N/H \cong W$. Moreover, the group $B$ is a compact open subgroup, and every standard parabolic subgroup $P_J = B W_J B$ for some $J \subseteq S$ is thus open in $G$. Important to our later purposes is the fact that the group $G$ acts transitively on the  \emph{complete} apartment system of $\Delta$. In particular $B$ acts transitively on the apartments containing $C_0$.

For a root $\alpha\in\Phi$, we denote as before  the unique reflection of $W$ fixing the wall $\partial\alpha$ pointwise by $r_\alpha$. In addition, we choose some element $n_\alpha \in  N \cap \langle U_{\alpha}\cup U_{-\alpha}\rangle$ which maps onto $r_{\alpha}$ under the quotient map $N \to N/H \cong W$.

\medskip
Before we state a more precise version of Theorem~\ref{main thm intro}, we will need some additional results on the BN-pair structure of $G$. This is the object of the following paragraph.

%%%%%%%%%%%%%%%%%%%%%%%%%%%%%%%%%%%%%%%%%%%%%%%%%%%%%%%%%%%%%%%%%%%
\subsection{On Levi decompositions in complete Kac--Moody groups}
%%%%%%%%%%%%%%%%%%%%%%%%%%%%%%%%%%%%%%%%%%%%%%%%%%%%%%%%%%%%%%%%%%%

Given $J\subseteq S$, we denote by $\mathcal P_J= \mathcal B_+ W_J \mathcal B_+$ (resp. $P_J = B W_J B$) the standard parabolic subgroup of $\GGG$ (resp. $G$) of type $J$ and by $R_J(C_0)$ the $J$-residue of $\Delta$ containing the chamber $C_0$. Thus $\mathcal P_J = \Stab_\GGG(R_J(C_0))$, $P_J = \Stab_G(R_J(C_0))$ and $\mathcal P_J$ is dense in $P_J$. 

We further set  $\Phi_J=\{\alpha\in\Phi \ | \ \exists v\in W_J, \ s\in J: \alpha=v\alpha_s\}$ and
$$
\mathcal L_J^+  = \la U_{\alpha} \; | \; \alpha\in\Phi_J\ra.
$$
Finally, we set $\mathcal L_J=\mathcal H\cdot\mathcal L_J^+$ and denote by $\mathcal U_J$ the normal closure of $\la U_{\alpha} \;| \; \alpha\in\Phi, \ \alpha\supset R_{J}(C_0)\cap A_0\ra$ in $\mathcal B_+$. Following  \cite[6.2.2]{theseBR}, there is  a semidirect decomposition 
$$
\mathcal P_J = \mathcal L_J \ltimes \mathcal U_J.
$$ 
The group $\mathcal U_J$ is called the \textbf{unipotent radical} of the parabolic subgroup $\mathcal P_J$, and $\mathcal L_J$ is called the  \textbf{Levi factor}. 

We next define 
$$
L_J^+ = \overline{\mathcal L_J^+}, 
\hspace{1cm}
L_J = \overline{\mathcal L_J}
\hspace{1cm} \text{and} \hspace{1cm}
U_J = \overline{\mathcal U_J}.
$$ 
Thus $U_J$ and $L_J$ are closed subgroups of $P_J$, respectively called the  \textbf{unipotent radical} and  the  \textbf{Levi factor}.

\begin{lemma}\label{lemme these BR}
We have the following:
\begin{enumerate}[(i)]
\item $U_J$ is a compact normal subgroup of $P_J$, and we have $P_J = L_J\cdot U_J$. 

\item 
$L_J^+ $ is normal in $L_J$ and we have $L_J = \mathcal H\cdot L_J^+$. 

\end{enumerate}
\end{lemma}
\begin{proof}
Since $\mathcal U_J$ is normal in $\mathcal P_J$, which is dense in $P_J$, it is clear that $U_J$ is normal in $P_J$. Moreover $U_J$ is compact (since it is contained in $B$) and the product $L_J\cdot U_J$ is thus closed in $P_J$.  Assertion (i) follows since $L_J\cdot U_J$ contains $\mathcal P_J$. 

For assertion~(ii), we remark that $\mathcal H$ normalizes $\mathcal L_J^+$ and hence also $L_J^+$. Moreover, since $\mathcal H$ is finite, hence compact, the product $\mathcal H\cdot L_J^+$ is closed. Since $\mathcal H\cdot \mathcal L_J^+$ is dense in $L_J$, the conclusion follows.
\end{proof}

Remark that the decomposition $P_J = L_J\cdot U_J$ is even semidirect when $J$ is spherical, see \cite[section 1.C.]{RemyRonan}. It is probably also the case in general, but this will not be needed here.

\begin{lemma}\label{lemme conclusion 2}
Let $J\subseteq S$. Then every open subgroup $O$ of $P_J$ that contains the product $L_J^+\cdot U_{J\cup J^{\perp}}$ has finite index in $P_J$.
\end{lemma}
\begin{proof}
Set $K:=J^{\perp}$ and $U:=U_{J\cup J^{\perp}}$. Note that $U\triangleleft P_{J\cup K}=L_{J\cup K}\cdot U$. Moreover, $L_J^+$ is normal in $L_{J\cup K}$. Indeed, as $[U_{\alpha},U_{\beta}]=1$  for all $\alpha\in\Phi_J$ and $\beta\in\Phi_K$, the subgroups $\mathcal L_J^+$ and $\mathcal L_K^+$ centralize each other. Since in addition $\mathcal H$ normalizes each root group, we get a decomposition $\mathcal L_{J\cup K}=\mathcal H\cdot \mathcal L_J^+\cdot \mathcal L_K^+$. In particular, $\mathcal L_{J\cup K}$ normalizes $\mathcal L_J^+$, whence also $L_J^+$. As the normalizer of a closed subgroup is closed, this implies that $L_{J\cup K}$ normalizes $L_J^+$, as desired.

Let $\pi_1\co P_{J\cup K}\to P_{J\cup K}/U$ denote the natural projection. Then $\pi_1(L_J^+)$ is normal in $P_{J\cup K}/U$, since it is the image of $L_J^+$ under the composition map $$L_{J\cup K}\to \frac{L_{J\cup K}}{L_{J\cup K}\cap U}\stackrel{\cong}{\to}\frac{P_{J\cup K}}{U}: l\mapsto l(L_{J\cup K}\cap U)\mapsto lU.$$ 
Let $\pi\co P_{J\cup K}\to \pi_1(P_{J\cup K})/\pi_1(L_J^+)$ denote the composition of $\pi_1$ with the canonical projection onto $\pi_1(P_{J\cup K})/\pi_1(L_J^+)$. Note that $\pi$ is an open continuous group homomorphism. Then $\pi(P_J)=\pi_1(L_J^+\cdot U_J\cdot\mathcal H)/\pi_1(L_J^+)$ is compact. Indeed, it is homeomorphic to the quotient of the compact group $\pi_1(U_J\cdot\mathcal H)$ by the normal subgroup $\pi_1(L_J^+\cap U_J\cdot\mathcal H)$ under the map $$\frac{\pi_1(L_J^+\cdot U_J\cdot\mathcal H)}{\pi_1(L_J^+)}\stackrel{\cong}{\to} \frac{\pi_1(U_J\cdot\mathcal H)}{\pi_1(L_J^+\cap U_J\cdot\mathcal H)}:\pi_1(l\cdot u)\pi_1(L_J^+)\mapsto \pi_1(u)\pi_1(L_J^+\cap U_J\cdot\mathcal H).$$ 
In particular, since $\pi(O)$ is open in $\pi(P_J)$, it has finite index in $\pi(P_J)$. But then since $O=\pi\inv(\pi(O))$ by hypothesis, $O$ has finite index in $\pi\inv(\pi(P_J))=P_J$, as desired.
\end{proof}

%%%%%%%%%%%%%%%%%%%%%%%%%%%%%%%%%%%%%%%%%%%%%%%%%%%%%%%%%%%%%%%%%%%%%%
\subsection{A refined version of Theorem~\ref{main thm intro}}
%%%%%%%%%%%%%%%%%%%%%%%%%%%%%%%%%%%%%%%%%%%%%%%%%%%%%%%%%%%%%%%%%%%%%%

We will prove the following statement, having Theorem~\ref{main thm intro} as an immediate corollary.

\begin{thm}\label{thm complet}
Let $O$ be an open subgroup of $G$. 
Let $J\subseteq S$ be the type of a residue which is stabilized by some finite index subgroup of $O$ and minimal with respect to this property. 

Then there exist a spherical subset $J'\subseteq J^{\perp}$ and an element $g\in G$ such that 
$$L_J^+\cdot U_{J\cup J^{\perp}}<gOg\inv<P_{J\cup J'}.$$ 
In particular, $gOg\inv$ has finite index in $P_{J\cup J'}$.

Moreover, any subgroup of $G$ containing $gOg\inv$ as a finite index subgroup is contained in $P_{J\cup J''}$ for some spherical subset $J''\subseteq J^{\perp}$. In particular, only finitely many distinct parabolic subgroups contain $O$ as a finite index subgroup.
\end{thm}

%%%%%%%%%%%%%%%%%%%%%%%%%%%%%%%%%%%%%%%%%%%%%%%%%%
\subsection{Proof of Theorem~\ref{thm complet}: outline and first observations}
%%%%%%%%%%%%%%%%%%%%%%%%%%%%%%%%%%%%%%%%%%%%%%%%%%

This section and the next ones are devoted to the proof of Theorem~\ref{thm complet} itself. 

Let thus $O$ be an open subgroup of $G$. We define the subset $J$ of $S$ as in the statement of the theorem, namely, $J$ is minimal amongst the subsets $L$ of $S$ for which there exists a $g\in G$ such that $O\cap g\inv P_Lg$ has finite index in $O$.   For such a $g \in G$, we set $O_1 = g O g\inv \cap P_J$. Thus $O_1$ stabilizes $R_J(C_0)$ and is an open subgroup of $G$ contained in $gOg\inv$ with  finite index.

\medskip
We first observe that the desired statement is essentially empty when $O$ is compact. Indeed, in that case  the Bruhat--Tits fixed point theorem ensures that $O$ stabilizes a spherical residue of $G$, and hence Theorem~\ref{thm complet} stands proven with $J=\varnothing$. It thus remains to prove the theorem when $O$, and hence also $O_1$, is non-compact, which we assume henceforth.

\medskip
Recall from the previous section that we call a subset $J\subseteq S$ {\bf essential} if all its irreducible components are non-spherical. We begin with the following simple observation.

\begin{lemma}\label{lemme irr non sph}
$J$ is essential.
\end{lemma}
\begin{proof}
Let $J_1\subseteq J$ denote the union of the non-spherical irreducible components of $J$. As $P_{J_1}$ has finite index in $P_J$,  the subgroup $O_1\cap P_{J_1}$ is open of finite index in $O_1$ and stabilizes $R_{J_1}(C_0)$. The definition of $J$ then yields $J_1=J$.
\end{proof}

\medskip

Let us now describe the outline of  the proof. Our first task will be to show that $O_1$ contains $L_J^+$. We will see that this is equivalent to prove that $O_1$ acts transitively on the standard $J$-residue $R_J(C_0)$, or else that the stabilizer in $O_1$ of any apartment $A$ containing $C_0$ is transitive on $R_J(C_0)\cap A$. Since each group $\Stab_{O_1}(A)/\Fix_{O_1}(A)$ can be identified with a subgroup of the Coxeter group $W$ acting on $A$, we will be in a position to apply the results on Coxeter groups from the previous section. This will allow us to show that each $\Stab_{O_1}(A)/\Fix_{O_1}(A)$ contains a finite index parabolic subgroup of type $I_{A} \subseteq J$, and hence acts transitively on the corresponding residue.

We thus begin by defining some ``maximal" subset $I$ of $J$ such that $\Stab_{O_1}(A_1)$ acts transitively on $R_I(C_0)\cap A_1$ for a suitably chosen apartment $A_1$ containing $C_0$. We then establish that $I$ contains all the types $I_A$ when $A$ varies over all apartments containing $C_0$. This eventually allows us to prove that in fact $I=J$, so that $\Stab_{O_1}(A_1)$ is transitive on $R_J(C_0)\cap A_1$, or else that $O_1$ contains $L_J^+$, as desired.

We next show that $O_1$ contains the unipotent radical $U_{J \cup J^\perp}$. Finally, we make use of the transitivity of $O_1$ on $R_J(C_0)$ to prove that $O$ is contained in the desired parabolic subgroup.

%%%%%%%%%%%%%%%%%%%%%%%%%%%%%%%%%%%%%%%%%%%%%%%%%%
\subsection{Proof of Theorem~\ref{thm complet}: $O_1$ contains $L_J^+$}
%%%%%%%%%%%%%%%%%%%%%%%%%%%%%%%%%%%%%%%%%%%%%%%%%%

We first need to introduce some additional notation which we will retain until the end of the proof.

\medskip 
Let $\AAA_{\geq C_0}$ denote the set of apartments of $\Delta$ containing $C_0$. For $A\in \AAA_{\geq C_0}$, set 
$N_A:=\Stab_{O_1}(A)$ and $\overline{N}_A=N_A/\Fix_{O_1}(A)$, which one identifies with a subgroup of $W$. Finally, for $h\in N_A$, denote by $\overline{h}$ its image in $\overline{N}_A\leq W$. Here is the main tool developed in the previous section.

\begin{lemma}\label{lemme parabolique essentiel}
For all $A\in \AAA_{\geq C_0}$, there exists $h\in N_A$ such that 
$$
\Pc(\overline{h})=\langle r_{\alpha} \ | \ \textrm{$\alpha$ $\overline{h}$-essential root of $\Phi$}\rangle
$$ 
and is of finite index in $\Pc(\overline{N}_A)$. 
\end{lemma}
\begin{proof}
This is an immediate consequence of Corollary~\ref{cor thm Coxeter} and Lemma~\ref{lemme Krammer}.
\end{proof}

\begin{lemma}\label{lemme subsequence}
Let $(g_n)_{n\in\NN}$ be an infinite sequence of elements of $O_1$. Then there exist an apartment $A\in \AAA_{\geq C_0}$, a subsequence $(g_{\psi(n)})_{n\in\NN}$ and elements $z_n\in O_1$, $n\in \NN$, such that for all $n\in \NN$ we have
\begin{itemize}
\item[(1)]
$h_n:=z_{0}\inv z_n\in N_A$,
\item[(2)]
$\dist(C_0,z_nR)=\dist(C_0,g_{\psi(n)}R)$ for every residue $R$ containing $C_0$ and
\item[(3)]
$|\dist(C_0,h_nC_0)-\dist(C_0,g_{\psi(n)}C_0)|<\dist(C_0,z_{0}C_0)$.
\end{itemize}
\end{lemma}
\begin{proof}
As $O_1$ is open, it contains a finite index subgroup $K:=\Fix_G(B(C_0,r))$ of $B$ for some $r\in \NN$. Since $B$ is transitive on the set $\AAA_{\geq C_0}$, we deduce that $K$ has only finitely many orbits in $\AAA_{\geq C_0}$, say $\AAA_1,\dots,\AAA_k$. So, up to choosing a subsequence, we may assume that all chambers $g_nC_0$ belong to the same $K$-orbit $\AAA_{i_0}$ of apartments. Hence there exist elements $x_n\in K\subset O_1$ and an apartment $A'\in\AAA_{i_0}$ containing $C_0$ such that $g_n':=x_ng_n\in O_1$, $g_n' C_0\in A'$ and $\dist(C_0,g_n' C_0)=\dist(C_0,g_nC_0)$. For each $n$, we now choose an element of $G$ stabilizing $A'$ and mapping $C_0$ to $g_n'C_0$. Thus such an element is in the same right coset modulo $B$ as $g_n'$. In particular, up to choosing a subsequence, we may assume it has the form $g_n'y_nb\in\Stab_G(A')$ for some $y_n\in K$ and some $b\in B$ independant of $n$. Denote by $\{\psi(n) \ | \ n\in\NN\}$ the resulting indexing set for the subsequence. Then setting $A:=bA'\in\AAA_{\geq C_0}$, the sequence $z_n:=g_{\psi(n)}'y_{\psi(n)}\in O_1$ is such that $h_n:=z_0\inv z_n\in b\Stab_{G}(A')b\inv\cap O_1=\Stab_{O_1}(A)=N_A$ and 
\begin{align*}
|\dist(C_0,h_nC_0)-\dist(C_0,g_{\psi(n)}C_0)|&=|\dist(z_0C_0,z_nC_0)-\dist(C_0,z_nC_0)|<\dist(C_0,z_0C_0).
\end{align*}
\end{proof}

\begin{lemma}\label{lemme existe orbite non bornee}
There exists an apartment $A\in \AAA_{\geq C_0}$ such that the orbit $N_{A}\cdot C_0$ is unbounded. In particular, the parabolic closure in $W$ of $\overline{N}_A$ is non-spherical.
\end{lemma}
\begin{proof}
Since $O_1$ is non-compact, the orbit $O_1\cdot C_0$ is unbounded in $\Delta$. For $n\in\NN$, choose $g_n\in O_1$ such that $\dist(C_0,g_nC_0)\geq n$. Then by Lemma~\ref{lemme subsequence}, there exist an apartment $A\in \AAA_{\geq C_0}$ and elements $h_n\in N_A$ for $n$ in some unbounded subset of $\NN$ such that $\dist(C_0,h_nC_0)$ is arbitrarily large when $n$ varies. This proves the lemma.
\end{proof}

Let $A_1\in \AAA_{\geq C_0}$ be an apartment such that the type of the product of the non-spherical irreducible components of $\Pc(\overline{N}_{A_1})$ is nonempty and maximal for this property. Such an apartment exists by Lemma~\ref{lemme existe orbite non bornee}. Now choose $h_{A_1}\in N_{A_1}$ as in Lemma~\ref{lemme parabolique essentiel}, so that in particular $[\Pc(\overline{N}_{A_1}):\Pc(\overline{h}_{A_1})]<\infty$. Up to conjugating $O_1$ by an element of $P_J$, we may then assume without loss of generality that $\Pc(\overline{h}_{A_1})$ is standard, non-spherical, and has essential type $I$. Moreover, it is maximal in the following sense: if $A\in \AAA_{\geq C_0}$ is such that $\Pc(\overline{N}_{A})$ contains a parabolic subgroup of essential type $I_A$ with $I_A\supseteq I$, then $I=I_A$.

Now that $I$ is defined, we need some tool to show that $O_1$ contains sufficiently many root groups $U_{\alpha}$. This will ensure that $O_1$ is ``transitive enough" in two ways: first on residues in the building by showing it contains subgroups of the form $\mathcal L_T^+$, and second on residues in apartments by establishing the presence in $O_1$ of enough $n_{\alpha}\in\langle U_{\alpha}\cup U_{-\alpha}\rangle$, since these lift reflections $r_{\alpha}$ in stabilizers of apartments. This tool is provided by the so-called (FPRS) property from \cite[2.1]{CaRe}, which we now state. Note for this that as $O_1$ is open, it contains the fixator in $G$ of a ball of $\Delta$: we fix $r\in\NN$ such that $O_1\supset K_r:=\Fix_G(B(C_0,r))$.

\begin{lemma}\label{lemme BR-PEC}
There exists a constant $N=N(W,S,r)\in\NN$ such that for every root $\alpha\in\Phi$ with $\dist(C_0,\alpha)>N$, the root group $U_{-\alpha}$ is contained in $\Fix_G(B(C_0,r))=K_r$.
\end{lemma}
\begin{proof}
See \cite[Prop. 4]{CaRe}.
\end{proof}

We also record a version of this result in a slightly more general setting.
\begin{lemma}\label{lemme BR-PEC generique}
Let $g\in G$ and let $A\in\AAA_{\geq C_0}$ containing the chamber $D:=gC_0$. Also, let $b\in B$ such that $A=bA_0$, and let $\alpha=b\alpha_0$ be a root of $A$, with $\alpha_0\in \Phi$. Then there exists $N=N(W,S,r)\in\NN$ such that if $\dist(D,-\alpha)>N$ then $bU_{\alpha_0}b\inv\subseteq gK_rg\inv$.
\end{lemma}
\begin{proof}
Take for $N=N(W,S,r)$ the constant of Lemma~\ref{lemme BR-PEC} and suppose that $\dist(D,-\alpha)>N$.
Let $h\in\Stab_G(A_0)$ be such that $hC_0=b\inv D$. Then $$N<\dist(D,-\alpha)=\dist(bh C_0,-b\alpha_0)=\dist(h C_0,-\alpha_0)=\dist(C_0,-h\inv\alpha),$$
and so Lemma~\ref{lemme BR-PEC} implies $h\inv U_{\alpha_0}h=U_{h\inv\alpha_0}\subseteq K_r$. Let $b_1\in B$ such that $bh=gb_1$. Then $$bU_{\alpha_0}b\inv\subseteq bhK_rh\inv b\inv= gb_1K_rb_1\inv g\inv=gK_rg\inv.$$
\end{proof}

This will prove especially useful in the following form, when we will use the description of the parabolic closure of some $w\in W$ in terms of $w$-essential roots as in Lemma~\ref{lemme parabolique essentiel}.

\begin{lemma}\label{lemme contient radiciels}
Let $A\in \AAA_{\geq C_0}$ and $b\in B$ such that $A=bA_0$. Also, let $\alpha=b\alpha_0$ ($\alpha_0\in\Phi$) be a $w$-essential root of $A$ for some $w\in\Stab_G(A)/\Fix_G(A)$, and let $g\in\Stab_G(A)$ be a representative of $w$. Then there exists $n \in\ZZ$ such that for $\epsilon\in\{+,-\}$ we have $U_{\epsilon\alpha_0}\subseteq b\inv g^{\epsilon n}K_rg^{-\epsilon n}b$.
\end{lemma}
\begin{proof}
Choose $n\in\ZZ$ such that $\dist(g^{\epsilon n}C_0,-\epsilon\alpha)>N$ for $\epsilon\in\{+,-\}$, where $N=N(W,S,r)$ is the constant appearing in the statement of Lemma~\ref{lemme BR-PEC}. Thus, for $\epsilon\in\{+,-\}$ we have $\dist(b\inv g^{\epsilon n}C_0,-\epsilon\alpha_0)>N,$
and so $\dist(C_0,-\epsilon(b\inv g^{-\epsilon n}b)\alpha_0)>N$. Lemma~\ref{lemme BR-PEC} then yields $$(b\inv g^{-\epsilon n}b)U_{\epsilon\alpha_0}(b\inv g^{-\epsilon n}b)\inv = U_{\epsilon (b\inv g^{-\epsilon n}b)\alpha_0}\subseteq K_r,$$ and so $$U_{\epsilon\alpha_0}\subseteq (b\inv g^{\epsilon n}b)K_r(b\inv g^{-\epsilon n}b)=(b\inv g^{\epsilon n})K_r(g^{-\epsilon n}b).$$
\end{proof}

We are now ready to prove how the different transitivity properties of $O_1$ are related.

\begin{lemma}\label{lemme equivalence transitive}
Let $T\subseteq S$ be essential, and let $A\in \AAA_{\geq C_0}$. Then the following are equivalent:
\begin{itemize}
\item[(1)]
$O_1$ contains $\mathcal L_T^+$;
\item[(2)]
$O_1$ is transitive on $R_T(C_0)$;
\item[(3)]
$N_A$ is transitive on $R_T(C_0)\cap A$;
\item[(4)]
$\overline{N}_A$ contains the standard parabolic subgroup $W_{T}$ of $W$.
\end{itemize}
\end{lemma}
\begin{proof}
The equivalence $(3)\Leftrightarrow (4)$, as well as the implications $(1)\Rightarrow (2),(3)$ are trivial. 

To see that $(4)\Rightarrow (2)$, note that if $b\in B$ maps $A_0$ onto $A$, then for each $\alpha_0\in\Phi_{T}$, we have $bU_{\pm \alpha_0}b\inv\subseteq O_1$, and so $O_1\supseteq b\mathcal L_T^+b\inv$ is transitive on $R_T(C_0)$. Indeed, let $\alpha_0\in\Phi_T$ and consider the corresponding root $\alpha:=b\alpha_0\in\Phi_T(A)$ of $A$. By Lemma~\ref{lemme racine essentielle}, there exists $w\in W_T\subseteq \overline{N}_A$ such that $\alpha$ is $w$-essential. Then if $g\in O_1$ is a representative for $w$, Lemma~\ref{lemme contient radiciels} yields an $n\in\ZZ$ such that for $\epsilon\in\{+,-\}$ we have $U_{\epsilon\alpha_0}\subseteq b\inv g^{\epsilon n}K_rg^{-\epsilon n}b\subseteq b\inv O_1 b$.

Finally, we show $(2)\Rightarrow (1)$. Again, it is sufficient to check that if $\alpha\in\Phi_T$, then $O_1$ contains $U_{\epsilon\alpha}$ for $\epsilon\in\{+,-\}$. By Lemma~\ref{lemme racine essentielle}, there exists $g\in\Stab_G(A_0)$ stabilizing $R_T(C_0)\cap A_0$ such that $\alpha$ is $\overline{g}$-essential, where $\overline{g}$ denotes the image of $g$ in the quotient group $\Stab_G(A_0)/\Fix_G(A_0)$. Then, by Lemma~\ref{lemme contient radiciels}, one can find an $n\in\ZZ$ such that $U_{\epsilon\alpha}\subseteq g^{\epsilon n}K_rg^{-\epsilon n}$ for $\epsilon\in\{+,-\}$. Now, since $O_1$ is transitive on $R_T(C_0)$, there exist $h_{\epsilon}\in O_1$ such that $h_{\epsilon} C_0= g^{\epsilon n} C_0$, and so we find $b_{\epsilon}\in B$ such that $g^{\epsilon n}=h_{\epsilon}b_{\epsilon}$. Therefore $$U_{\epsilon\alpha}\subseteq h_{\epsilon}b_{\epsilon}K_rb_{\epsilon}\inv h_{\epsilon}\inv=h_{\epsilon}K_r h_{\epsilon}\inv\subseteq O_1.$$
\end{proof}

Now, to ensure that $O_1$ indeed satisfies one of those properties for some ``maximal $T$'', we use Lemma~\ref{lemme parabolique essentiel} to show that stabilizers in $O_1$ of apartments contain finite index parabolic subgroups.

\begin{lemma}\label{lemme existence indice fini}
Let $A\in \AAA_{\geq C_0}$. Then there exists $I_A\subseteq S$ such that $\overline{N}_A$ contains a parabolic subgroup $P_{I_A}$ of $W$ of type $I_A$ as a finite index subgroup. 
\end{lemma}
\begin{proof}
Choose $h\in N_A$ as in Lemma~\ref{lemme parabolique essentiel}, so that in particular $\Pc(\overline{h})$ is generated by the reflections $r_{\alpha}$ with $\alpha$ an $\overline{h}$-essential root of $A$. Let $\alpha=b\alpha_0$ be such a root ($\alpha_0\in\Phi$), where $b\in B$ maps $A_0$ onto $A$. By Lemma~\ref{lemme contient radiciels}, we then find $K\in\ZZ$ such that for $\epsilon\in\{+,-\}$, $$U_{\epsilon\alpha_0}\subseteq (b\inv h^{\epsilon K})K_r(h^{-\epsilon K}b)\subseteq b\inv O_1b.$$ In particular, $n_{\alpha_0}\in\langle U_{\alpha_0}\cup U_{-\alpha_0}\rangle\subseteq b\inv O_1 b$. As $r_{\alpha_0}$ is the image in $W$ of $n_{\alpha_0}$ and since $r_{\alpha}=br_{\alpha_0}b\inv$, we finally obtain $\Pc(\overline{h})\subseteq \overline{N}_A$. Then $P_{I_A}:=\Pc(\overline{h})$ is the desired parabolic subgroup, of type $I_A$.
\end{proof}

For each $A\in \AAA_{\geq C_0}$, we fix such an $I_A\subseteq S$ which, without loss of generality, we assume essential. We also consider the corresponding parabolic $P_{I_A}$ contained in $\overline{N}_A$. Note then that $P_{I_{A_1}}$ has finite index in $\Pc(\overline{N}_{A_1})$ by Lemma~\ref{lemme parabolique indice fini}, and so $I=I_{A_1}$.

\begin{lemma}\label{lemme O contient L_I}
$O_1$ contains $L_I^+$.
\end{lemma}
\begin{proof}
As noted above, we have $I=I_{A_1}$ and $P_I=W_I$. Since $O_1$ is closed in $G$, Lemma~\ref{lemme equivalence transitive} allows us to conclude.
\end{proof}

We now have to show that $I$ is ``big enough'', that is, $I=J$. For this, we first need to know that $I$ is ``uniformly" maximal amongst all apartments containing $C_0$. 

\begin{lemma}\label{lemme I max}
Let $A\in \AAA_{\geq C_0}$. Then $I_A\subseteq I$.
\end{lemma}
\begin{proof}
Set $R_1:=R_I(C_0)\cap A$ and let $R_2$ be an $I_A$-residue in $A$ on which $N_A$ acts transitively and that is at minimal distance from $R_1$ amongst such residues. Note that $N_A$ is transitive on $R_1$ as well by Lemma~\ref{lemme equivalence transitive}. 

If $R_1\cap R_2$ is nonempty, then $N_A$ is also transitive on the standard $I\cup I_A$-residue of $A$ and so $\overline{N}_{A}$ contains $W_{I\cup I_A}$. By maximality of $I$ and since $I\cup I_A$ is again essential, this implies $I_A\subseteq I$, as desired.

We henceforth assume that $R_1\cap R_2=\varnothing$. Let $b\in B$ such that $bA_0=A$. Consider a root $\alpha=b\alpha_0$ of $A$, $\alpha_0\in\Phi$, whose wall $\partial\alpha$ separates $R_1$ from $R_2$. 

If both $R_1$ and $R_2$ are at unbounded distance from $\partial\alpha$, then the transitivity of $N_A$ on $R_1$ and $R_2$ together with Lemma~\ref{lemme BR-PEC generique} yield $bU_{\pm\alpha_0}b\inv\subseteq K_r\subseteq O_1$. Since $r_{\alpha_0}\in\langle U_{\alpha_0}\cup U_{-\alpha_0}\rangle$, we thus have $r_{\alpha}:=br_{\alpha_0}b\inv\in O_1$ and so $r_{\alpha}\in \overline{N}_A$. But then $\overline{N}_A=r_{\alpha}\overline{N}_Ar_{\alpha}\inv$ is also transitive on the $I_A$-residue $r_{\alpha}R_2$ which is closer to $R_1$, a contradiction.

If $R_2$ is at bounded distance from $\partial\alpha$ then by Lemma~\ref{lemme 8 murs}, $r_{\alpha}$ centralizes the stabilizer $P$ in $W$ of $R_2$, that is, $P=r_{\alpha}Pr_{\alpha}\inv$. Note that $\overline{N}_{A}$ contains $P$ since it is transitive on $R_2$. Thus $N_A$ is transitive on the $I_A$-residue $r_{\alpha}R_2$, which is closer to $R_1$, again a contradiction.

Thus we are left with the case where $R_1$ is contained in a tubular neighbourhood of every wall $\partial\alpha$ separating $R_1$ from $R_2$. But in that case, Lemma~\ref{lemme 8 murs} again yields that $W_I$ is centralized by every reflection $r_{\alpha}$ associated to such walls. Choose chambers $C_i$ in $R_i$, $i=1,2$, such that $\dist(C_1,C_2)=\dist(R_1,R_2)$, and let $\partial\alpha_1,\dots,\partial\alpha_k$ be the walls separating $C_1$ from $C_2$, crossed in that order by a minimal gallery from $C_1$ to $C_2$. Then each $\alpha_i$, $1\leq i\leq k$, separates $R_1$ from $R_2$ and so $w:=r_{\alpha_k}\dots r_{\alpha_1}$ centralizes $W_I$ and maps $C_1$ to $C_2$. So $W_I=wW_Iw\inv\subseteq \overline{N}_A$ is transitive on $wR_1$ and $R_2$, and hence also on $R_{I\cup I_A}(C_2)\cap A$. Therefore $\overline{N}_{A}$ contains a parabolic subgroup of essential type $I\cup I_A$, so that $I\supseteq I_A$ by maximality of $I$, as desired.
\end{proof}

\begin{lemma}\label{lemme W_I partout}
Let $A\in \AAA_{\geq C_0}$. Then $\overline{N}_{A}$ contains $W_{I}$ as a subgroup of finite index.
\end{lemma}
\begin{proof}
We know by Lemmas~\ref{lemme equivalence transitive} and~\ref{lemme O contient L_I} that $\overline{N}_{A}$ contains $W_{I}$. Also, by Lemma~\ref{lemme existence indice fini}, $\overline{N}_{A}$ contains a finite index parabolic subgroup $P_{I_A}=wW_{I_A}w\inv$ of type $I_A$, for some $w\in W$. Since $I_A\subseteq I$ by Lemma~\ref{lemme I max}, we get $W_{I_A}\subseteq\overline{N}_{A}$ and so the parabolic subgroup $P:=W_{I_A}\cap wW_{I_A}w\inv$ has finite index in $W_{I_A}$. As $I_A$ is essential, \cite[Prop.2.43]{ABrown} then yields $P=W_{I_A}$ and so $W_{I_A}\subseteq wW_{I_A}w\inv$. Finally, since the chain $W_{I_A}\subseteq wW_{I_A}w\inv\subseteq w^2W_{I_A}w^{-2}\subseteq\dots$ stabilizes, we find that $W_{I_A}=P_{I_A}$ has finite index in $\overline{N}_{A}$. The result follows.
\end{proof}

We are now ready to make the announced connection between $I$ and $J$.

\begin{lemma}\label{lemme I=J}
$I=J$.
\end{lemma}
\begin{proof}
Let $\mathcal{R}$ denote the set of $I$-residues of $\Delta$ containing a chamber of $O_1\cdot C_0$, and set $R:=R_I(C_0)$. We first show that the distance from $C_0$ to the residues of $\mathcal{R}$ is bounded, and hence that $\mathcal{R}$ is finite. 

Indeed, suppose for a contradiction that there exists a sequence of elements $g_n\in O_1$ such that $\dist(C_0,g_nR)\geq n$ for all $n\in\NN$. Then, up to choosing a subsequence and relabeling, Lemma~\ref{lemme subsequence} yields an apartment $A\in \AAA_{\geq C_0}$ and a sequence $(z_n)_{n\geq n_0}$ of elements of $O_1$ such that $h_n:=z_{n_0}\inv z_n\in N_A$ and $\dist(C_0,z_nR)=\dist(C_0,g_nR)$. Moreover by Lemma~\ref{lemme W_I partout}, we have a finite coset decomposition of the form $\overline{N}_A=\coprod_{j=1}^{t}{v_jW_I}$. Denote by $\pi\co N_A\to \overline{N}_A$ the natural projection. Again up to choosing a subsequence and relabeling, we may assume that $\pi(h_n)=v_{j_0}u_n$ for all $n\geq n_1$ (for some fixed $n_1\in\NN$), where each $u_n\in W_I$ and where $j_0$ is independant of $n$. Then the elements $w_n:=\pi(h_{n_1}\inv h_n)=\pi(z_{n_1}\inv z_n)$ belong to $W_I$. Thus the chambers $z_{n_1}C_0$ and $z_nC_0$ belong to the same $I$-residue since $z_{n_1}$ maps an $I$-gallery between $C_0$ and $w_nC_0$ to an $I$-gallery between $z_{n_1}C_0$ and $z_nC_0$. Therefore $$\dist(C_0,g_nR)=\dist(C_0,z_nR)\leq \dist(C_0,z_{n_1}C_0)$$ and so $\dist(C_0,g_nR)$ is bounded, a contradiction.

So $\mathcal{R}$ is finite and is stabilized by $O_1$. Hence the kernel $O'$ of the induced action of $O_1$ on $\mathcal{R}$ is a finite index subgroup of $O_1$ stabilizing an $I$-residue. Up to conjugating by an element of $O_1$, we thus have $O'<P_I$ and $[O_1:O']<\infty$. Then $O'':=O_1\cap P_I$ is open and contains $O'$, and has therefore finite index in $O_1$. The definition of $J$ finally implies that $I=J$.
\end{proof}

In particular, Lemmas~\ref{lemme O contient L_I} and~\ref{lemme I=J} yield the following.
\begin{cor}\label{cor O contient L_I}
$O_1$ contains $L_J^+$.
\end{cor}

%%%%%%%%%%%%%%%%%%%%%%%%%%%%%%%%%%%%%%%%%%%%%%%%%%
\subsection{Proof of Theorem~\ref{thm complet}: $O_1$ contains the unipotent radical  $U_{J \cup J^\perp}$}
%%%%%%%%%%%%%%%%%%%%%%%%%%%%%%%%%%%%%%%%%%

To show that $O_1$ contains the desired unipotent radical, we again make use of the (FPRS) property. %and of the description of this radical in terms of root groups.

\begin{lemma}\label{lemme radical unipotent}
$O_1$ contains the unipotent radical $U_{J\cup J^{\perp}}$.
\end{lemma}
\begin{proof}
By definition of $U_{J\cup J^{\perp}}$, we just have to check that for every $b\in B$ and every $\alpha\in\Phi$ containing $R_{J\cup J^{\perp}}(C_0)\cap A_0$, we have $bU_{\alpha}b\inv\in O_1$. Fix such $b$ and $\alpha$. In particular, $\alpha$ contains $R:=R_{J}(C_0)\cap A_0$. We claim that $R$ is at unbounded distance from the wall $\partial\alpha$ associated to $\alpha$. Indeed, if it were not, then as $J$ is essential by Lemma~\ref{lemme irr non sph}, the reflection $r_{\alpha}$ would centralize $W_J$ by Lemma~\ref{lemme 8 murs}, and hence would belong to $W_{J^{\perp}}$ by Lemma~\ref{lemme Deodhar}, contradicting $\alpha\supset R_{J^{\perp}}(C)\cap A_0$. 

Set now $A=bA_0$. Then $\alpha'=b\alpha$ is a root of $A$ containing $R':=R_J(C_0)\cap A$. Moreover, $R'$ is at unbounded distance from $-\alpha'$. Since $O_1$ is transitive on $R_J(C_0)$ by Corollary~\ref{cor O contient L_I}, there exists $g\in O_1$ such that $D:=gC_0\in R_J(C_0)\cap A$ and $\dist(D,-\alpha')>N$, where $N$ is provided by Lemma~\ref{lemme BR-PEC generique}. This lemma then implies that $bU_{\alpha}b\inv\subseteq gK_rg\inv\subseteq O_1$, as desired.
\end{proof}

%%%%%%%%%%%%%%%%%%%%%%%%%%%%%%%%%%%%%%%%%%%%%%%%%%
\subsection{Proof of Theorem~\ref{thm complet}: endgame}
%%%%%%%%%%%%%%%%%%%%%%%%%%%%%%%%%%%%%%%%%%

We can now prove that $gOg\inv$ is contained in a parabolic subgroup that has $P_J$ as a finite index subgroup.

\begin{lemma}\label{lemme conclusion 3}
Every subgroup $H$ of $G$ containing $O_1$ as a subgroup of finite index is contained in some standard parabolic $P_{J \cup J'}$ of type $J\cup J'$ , with $J'$ spherical and $J'\subseteq J^{\perp}$.
\end{lemma}
\begin{proof}
Recall that $O_1$ stabilizes the $J$-residue $R:=R_J(C_0)$ and acts transitively on its chambers by Corollary~\ref{cor O contient L_I}. Let $\mathcal{R}$ be the (finite) set of $J$-residues of $\Delta$ containing a chamber in the orbit $H\cdot C_0$. 

We first claim that for any $R'\in\mathcal{R}$ there is a constant $M$ such that $R$ is contained in an $M$-neighbourhood of $R'$ (and since $\mathcal{R}$ is finite we may then as well assume that this constant $M$ is independant of $R'$). Indeed, because $\mathcal{R}$ is finite, there is a finite index subgroup $H'$ of $H$ which stabilizes $R'$. In particular $\dist(D,R')=\dist(H'\cdot D,R')$ for any chamber $D$ of $R$. Moreover, the chambers of $R$ are contained in finitely $H'$-orbits since $H$ acts transitively on $R$. The claim follows. 

Let now $J'\subseteq S \setminus J$ be minimal such that $\textbf{R}:=R_{J\cup J'}(C_0)$ contains the reunion of the residues of $\mathcal{R}$. In other words, $H<P_{J\cup J'}$ with $J'$ minimal for this property.

We next show that $J'\subseteq J^{\perp}$. For this, it is sufficient to see that $H$ stabilizes $R_{J\cup J^{\perp}}(C_0)$. 

Note that, given $R'\in\mathcal{R}$, if $A$ is an apartment containing some chamber $C'_0$ of $R'$, then every chamber $D$ in $R\cap A$ is at distance at most $M$ from $R'\cap A$. Indeed, if $\rho=\rho_{A,C'_0}$ is the retraction of $\Delta$ onto $A$ centered at $C'_0$, then for every $D'\in R'$ such that $\dist(D,D')\leq M$, the chamber $\rho(D')$ belongs to $R'\cap A$ and is at distance at most $M$ from $D=\rho(D)$ since $\rho$ is distance decreasing (see \cite[Lemma 11.2]{MR1709955}). 

Let now $g\in H$ and set $R':=gR\in\mathcal{R}$. Let $\Gamma$ be a minimal gallery from $C_0$ to its combinatorial projection onto $R'$, which we denote by $C'_0$. Let $A$ be an apartment containing $\Gamma$. Finally, let $w\in W=\Stab_G(A)/\Fix_G(A)$ such that $wC_0=C'_0$. We want to show that $\Gamma$ is a $J^{\perp}$-gallery, that is, $w\in W_{J^{\perp}}$. 

To this end, we first observe that, since $\Gamma$ joins $C_0$ to its projection onto $R'$, it does not cross any wall of $R' \cap A$.  We claim that $\Gamma$ does not cross any wall of $R \cap A$ either. Indeed, assume on the contrary that  $\Gamma$ crosses some wall $m$ of $R \cap A$. Then by Lemma~\ref{lemme racine essentielle} we would find a wall $m'\neq m$ intersecting $R\cap A$ and parallel to $m$, and therefore also chambers of $R\cap A$ at unbounded distance from $R'\cap A$,  a contradiction. 
 
Thus every wall crossed by $\Gamma$ separates $R \cap A$ from $R' \cap A$. In particular, $R\cap A$ is contained in an $M$-neighbourhood of any such wall $m$ since it is contained in an $M$-neighbourhood of $R'\cap A$ and since every minimal gallery between a chamber in $R\cap A$ and a chamber in $R'\cap A$ crosses $m$. Then, by Lemmas~\ref{lemme Deodhar} and~\ref{lemme 8 murs}, the reflection associated to $m$ belongs to $ W_{J^{\perp}}$. Therefore $w$ is a product of reflections that belong to $W_{J^{\perp}}$, as desired.

Finally, we show that $J'$ is spherical. As $\textbf{R}$ splits into a product of buildings $\textbf{R}=R_J\times R_{J'}$, where $R_J:=R_J(C_0)$ and $R_{J'}:=R_{J'}(C_0)$, we get a homomorphism $H\to \Aut(R_J)\times\Aut(R_{J'})$. As $O_1$ stabilizes $R_J$ and has finite index in $H$, the image of $H$ in $\Aut(R_{J'})$ has finite orbits in $R_{J'}$. In particular, by the Bruhat--Tits fixed point theorem, $H$ fixes a point in the Davis realization of $R_{J'}$, and thus stabilizes a spherical residue of $R_{J'}$. But this residue must be the whole of $R_{J'}$ by minimality of $J'$. This concludes the proof of the lemma.
\end{proof}

\begin{proof}[Proof of Theorem~\ref{thm complet}] 
The first statement summarizes Corollary~\ref{cor O contient L_I} and Lemmas~\ref{lemme radical unipotent} and \ref{lemme conclusion 3}, since some conjugate $gOg\inv$ of $O$ contains $O_1$ as a finite index subgroup. The second statement then follows from Lemma~\ref{lemme conclusion 2} applied to the open subgroup $O_1$ of $P_J$.
Finally, the two last statements are a consequence of Lemma~\ref{lemme conclusion 3}. Indeed, any subgroup $H$ containing $gOg\inv$ with finite index also contains $O_1$ with finite index. Then $H$ is a subgroup of some standard parabolic $P_{J\cup J'}$ for some spherical subset $J' \subset J^\perp$. Moreover, since the index of $O_1$ in $P_{J \cup J'}$ is finite, and since there are only finitely many spherical subsets of $J^\perp$, it follows that there are only finitely many possibilities for $H$.
\end{proof}

\begin{rem}
Let $O$ be a subgroup of $G$, and let $J\subseteq S$ be as in the statement of Theorem~\ref{thm complet}. Assume that $J^\perp$ is spherical. Then $L_J^+ \cdot U_{J\cup J^\perp}$ has finite index in $P_{J\cup J^\perp}$ and is thus open since it is closed. Thus, in that case, $O$ is open if and only if $L_J^+ \cdot U_{J\cup J^\perp}<gOg\inv< P_{J\cup J^\perp} $ for some $g\in G$.
\end{rem}

\begin{cor}
Let $O$ be an open subgroup of $G$ and let $J\subseteq S$ be minimal such that $O$ virtually stabilizes a $J$-residue. If $J^{\perp}=\varnothing$, then there exists some $g\in G$ such that $L_J^+ \cdot U_J<gOg\inv<P_J=\mathcal H\cdot L_J^+ \cdot U_J $.
\end{cor}

\begin{proof}
This readily follows from Theorem~\ref{thm complet}.
\end{proof}

To prove Corollary~\ref{cor intro}, we use the following general fact, which is well known in the discrete case. 

\begin{lemma}\label{lem:Noeth}
Let $G$ be a locally compact group. 
Then $G$ is Noetherian if and only if every open subgroup is compactly generated. 
\end{lemma}

\begin{proof}
Assume that $G$ is Noetherian and let $O < G$ be open. Let $U_1 < O$ be the subgroup generated by some compact identity neighbourhood $V$ in $O$. If $U_1 \neq O$, there is some $g_1 \in O \setminus U_1$ and we let $U_2 = \langle U_1 \cup \{g_1 \} \rangle$. Proceeding inductively we obtain an ascending chain of open subgroups $U_1 < U_2 < \dots < O$, and the ascending chain condition ensures that $O = U_n$ for some $n$. In other words $O$ is generated by the compact set $V \cup \{g_1, \dots, g_n\}$. 

Assume conversely that every open subgroup is compactly generated, and let $U_1 < U_2 < \dots$ be an ascending chain of open subgroups. Then $U = \bigcup_n U_n$ is an open subgroup. Let $C$ be a compact generating set for $U$. By compactness, the inclusion $C \subset \bigcup_n U_n$ implies that $C$ is contained in $U_n$ for some $n$ since every $U_j$ is open. Thus $U  = \langle C \rangle < U_n$, whence $U = U_n$ and $G$ is Noetherian.
\end{proof}

\begin{proof}[Proof of Corollary~\ref{cor intro}]
By Theorem~\ref{main thm intro}, every open subgroup of a complete Kac--Moody group $G$ over a finite field is contained as a finite index subgroup in some parabolic subgroup. Notice that parabolic subgroups are compactly generated by the Svarc--Milnor Lemma since they act properly and cocompactly on the residue of which they are the stabilizer. Since a cocompact subgroup of a group acting cocompactly on a space also acts cocompactly on that space, it follows for the same reason that all open subgroups of $G$ are compactly generated; hence $G$ is Noetherian by Lemma~\ref{lem:Noeth}.
\end{proof}

\begin{proof}[Proof of  Corollary~\ref{corintro2}]
Immediate from Theorem~\ref{main thm intro} since Coxeter groups of affine and compact hyperbolic type are precisely those Coxeter groups all of whose proper parabolic subgroups are finite. 
\end{proof}

\bibliographystyle{amsalpha} 
\bibliography{OpenKM} 
\end{document}